\documentclass[final]{elsarticle}

\usepackage{amsmath,amssymb,amsthm,xspace,mathtools,booktabs,hyperref,color}
\usepackage{mathrsfs,verbatim,xspace}
\usepackage{epstopdf}
\allowdisplaybreaks

\journal{Journal of Complexity}

\newcommand{\dif}{{\rm d}}

\newlength{\overwdth}

\def\abs#1{\ensuremath{\left \lvert #1 \right \rvert}}

\newcommand{\norm}[2][{}]{\ensuremath{\left \lVert #2 \right \rVert}_{#1}}

\DeclareMathOperator{\cost}{cost}

\DeclareMathOperator{\sign}{sign}

\DeclareMathOperator{\Order}{{\mathcal O}}

\newcommand{\tL}{\widetilde{L}}

\newcommand{\reals}{\mathbb{R}}

\newcommand{\natzero}{\mathbb{N}_{0}}

\newcommand{\cb}{\mathcal{B}}
\providecommand{\cc}{\mathcal{C}}

\newcommand{\ci}{\mathcal{I}}

\newcommand{\cu}{\mathcal{U}}

\newcommand{\cw}{\mathcal{W}}

\newcommand{\fC}{\mathfrak{C}}
\newcommand{\fh}{\mathfrak{h}}

\newcommand{\abstol}{\varepsilon}
\newcommand{\zton}{0\!:\!n}
\newcommand{\oton}{1\!:\!n}
\newcommand{\datasites}{x_{0:n}}
\makeatletter
\newcommand{\vast}{\bBigg@{4}}
\newcommand{\Vast}{\bBigg@{6}}
\makeatother
\theoremstyle{definition}
\newtheorem*{algoA}{Algorithm $A$}
\newtheorem*{algoM}{Algorithm $M$}

\renewcommand{\cw}{W}

\definecolor{orange}{rgb}{1.0,0.3,0.0}
\definecolor{violet}{rgb}{0.75,0,1}
\definecolor{darkgreen}{rgb}{0,0.6,0}

\newenvironment{FJHchange}{}{} 

\newcommand{\Ixl}{I_{x,l}}

\newcommand{\tell}{\tilde{\ell}}

\newcommand{\chL}{\widehat{L}}

\newcommand{\tgamma}{\widetilde{\gamma}}
\newcommand{\hM}{\widehat{M}}
\DeclareMathOperator{\ninit}{ninit}
\DeclareMathOperator{\oerr}{\overline{err}}
\DeclareMathOperator{\herr}{\widehat{err}}
\newtheorem{theorem}{Theorem}

\newtheorem{prop}[theorem]{Proposition}
\newcommand{\funappxg}{\texttt{funappx\_g}\xspace}

\newcommand{\funming}{\texttt{funmin\_g}\xspace}
\newcommand{\fminbnd}{\texttt{fminbnd}\xspace}

\newcommand{\minfi}{\min\limits_{i \in 0:n} f(x_i)} 
\newcommand{\minfii}{\min(f(x_{i-1}), f(x_i))} 

\begin{document}

\begin{frontmatter}

\title{Local Adaption for Approximation and Minimization of Univariate Functions}
\author{Sou-Cheng T.~Choi}
\author{Yuhan Ding}
\author{Fred J.~Hickernell}
\address{Department of Applied Mathematics, Illinois Institute of Technology,
RE 208, 10 West 32$^{\text{nd}}$ Street, Chicago, Illinois, 60616,
USA}
\ead{hickernell@iit.edu}
\author{Xin Tong}
\address{Department of Mathematics, Statistics, and Computer Science, University of
Illinois at Chicago, Room 322 SEO, 851 S. Morgan Street, Chicago, Illinois, 60607,
USA}

\begin{abstract}
Most commonly used \emph{adaptive} algorithms for univariate real-valued function approximation and global
minimization lack theoretical guarantees. Our new locally adaptive algorithms
are guaranteed to provide answers that satisfy a user-specified absolute
error tolerance for a cone, $\cc$, of non-spiky input functions in the Sobolev space
$\cw^{2,\infty}[a,b]$. Our algorithms automatically determine where to sample the
function---sampling more densely where the second derivative is larger. The
computational cost of our algorithm for approximating a univariate function~$f$ on a
bounded interval with $L^{\infty}$-error no greater than~$\abstol$ is
$\smash{\Order\Bigl(\sqrt{\norm[\frac12]{f''}/\abstol}\Bigr)}$ as $\abstol \to 0$.
This is the
same order as that of the best function approximation algorithm for functions in~$\cc$.
The computational cost of our global minimization algorithm is of the same order
and the cost can be substantially less if~$f$
significantly exceeds its minimum over much of the domain. Our
Guaranteed Automatic Integration Library (GAIL) contains these new algorithms. We
provide numerical experiments to illustrate their superior performance.
\end{abstract}


\end{frontmatter}

\section{Introduction} \label{sec:intro}

Our goal is to reliably solve univariate function approximation and global minimization
problems by adaptive algorithms. We prescribe a suitable set, $\cc$, of continuously
differentiable, real-valued functions defined on a finite interval $[a,b]$.  Then, we
construct algorithms $A:(\cc,(0,\infty)) \to L^{\infty}[a,b]$ and $M:
(\cc,(0,\infty)) \to \reals$ such that for any $f \in \cc$ and any error
tolerance $\abstol > 0$,
\begin{gather} \norm{f - A(f,\abstol)} \le \abstol,
\tag{APP} \label{appxprob} \\
0 \le M(f,\abstol) - \min_{a \le x \le b} f(x) \le
\abstol. \tag{MIN} \label{optprob}
\end{gather}
Here, $\norm{\cdot}$ denotes the
$L^{\infty}$-norm on $[a,b]$, i.e., $\norm{f} =\sup_{x\in[a,b]} |f(x)|$. The algorithms~$A$ and~$M$ depend only on function
values.

Our algorithms proceed iteratively until their data-dependent stopping
criteria are satisfied. The input functions are sampled nonuniformly over
$[a,b]$, with the sampling density determined by the function data.
We call our algorithms \emph{locally adaptive}, to distinguish them from
globally adaptive algorithms that have a fixed sampling pattern and only the
sample size determined adaptively.

\subsection{Key Ideas in Our Algorithms} \label{subsec:keyideas}

Our algorithms $A$ and $M$ are based on a \emph{linear spline}, $S(f,\datasites)$
defined on $[a,b]$.  Let ${0\!:\!n}$ be shorthand for $\{0, \ldots, n\}$, and let
$\datasites$ be any
ordered sequence of $n+1$ points that includes the endpoints of the interval,
i.e., $a=:x_0 <x_1 < \cdots < x_{n-1} < x_{n}:=b$. We call such a sequence a
\emph{partition}. Then given any $\datasites$ and any $i \in \oton$, the linear
spline  is defined for $x \in [x_{i-1},x_i]$ by
\begin{align} \label{splinedef}
   S(f,\datasites)(x) :=
   \frac{x-x_i}{x_{i-1} - x_i} f(x_{i-1}) + \frac{x-x_{i-1}}{x_{i} - x_{i-1}}f(x_i).
\end{align}
The error of the linear spline is bounded in terms
of the second derivative of the input function as follows \cite[Theorem
3.3]{BurFaiBur16a}:
\begin{equation} \label{appxerrbda}
	\norm[{[x_{i-1},x_i]}]{f - S(f,\datasites)}
	\le \frac{(x_i - x_{i-1})^2\norm[{[x_{i-1},x_i]}]{f''}}{8}, \quad i \in \oton,
\end{equation}
where $\norm[{[\alpha,\beta]}]{f}$ denotes the $L^{\infty}$-norm of $f$ restricted to the
interval $[\alpha,\beta] \subseteq [a,b]$. This error bound leads us to focus on
input functions in the Sobolev space $\cw^{2,\infty}:= \cw^{2,\infty}[a,b] := \{f \in
C^1[a,b] : \norm{f''} < \infty \}$.

Algorithms $A$ and $M$ require upper bounds on $\norm[{[x_{i-1},x_i]}]{f''}$, $i \in
1\!:\!n$,
to make use of~\eqref{appxerrbda}. A nonadaptive algorithm might assume that
$\norm{f''} \le \sigma$, for some known~$\sigma$, and proceed to choose $n =
\bigl \lceil (b-a)\smash{\sqrt{\sigma/(8 \varepsilon)}} \, \bigr \rceil$, $x_i = a +
i(b-a)/n$, $i \in 0\!:\!n$. Providing an upper bound on $\norm{f''}$ is
often impractical, and so we propose adaptive algorithms that do not
require such information.

However, one must have some a priori information about $f \in \cw^{2,\infty}$
to construct successful algorithms for \eqref{appxprob} or \eqref{optprob}.
Suppose that algorithm $A$ satisfies \eqref{appxprob} for the zero function $f=0$, and
$A(0,\varepsilon)$ uses the data sites $x_{0:n}\subset [a,b]$. Then one can
construct a \emph{nonzero} function $g \in \cw^{2,\infty}$ satisfying $g(x_i) =
0$, $i\in \zton$ but with $\smash{\norm{g- A(g,\abstol)}} = \smash{\norm{g -
A(0,\abstol)}} > \varepsilon$.

Our set $\cc \subset \cw^{2,\infty}$ for which $A$ and $M$ succeed
includes only those functions whose second derivatives do not change dramatically
over a short distance. The precise definition of $\cc$ is given in Section
\ref{sec:cone}. This allows us to use second-order divided differences to
construct rigorous upper bounds on the linear spline error in
\eqref{appxerrbda}. These data-driven error bounds inform the stopping criteria
for Algorithm $A$ in Section \ref{subsec:appxalgo} and Algorithm $M$ in Section
\ref{sec:minalgo}.

The computational cost of Algorithm $A$ is analyzed in Section
\ref{subsec:appxcost} and is shown to be
$\smash{\Order\Bigl(\sqrt{\norm[\frac12]{f''}/\abstol}\Bigr)}$ as $\abstol \to 0$. Here,
$\norm[\frac12]{\cdot}$ denotes the $L^{\frac12}$-quasi-norm, a special case of the $L^{p}$-quasi-norm, $\norm[p]{f} :=
\bigl(\int_a^b \abs{f}^p \, \dif x \bigr)^{1/p}$, $0 < p < 1$. Since
$\norm[\frac12]{f''}$ can be much smaller than $\norm{f''}$, locally adaptive
algorithms can be more efficient than globally adaptive algorithms, whose
computational costs are proportional to $\sqrt{\norm{f''}/\abstol}$. The
computational complexity of \eqref{appxprob} is determined in Section
\ref{subsec:appxcomp} to be
$\Order\Bigl(\smash{\sqrt{\norm[\frac12]{f''}/\abstol}}\Bigr)$ as well.

The computational cost of our optimization algorithm $M$ is analyzed in
Section~\ref{subsec:optcost}. A lower bound on the computational complexity of
\eqref{optprob} is a subject for future investigation.

Our algorithms are implemented in our MATLAB~\cite{MAT9.1} Guaranteed Automatic
Integration Library (GAIL) \cite{ChoEtal15a}. Section \ref{sec:examples}
provides numerical examples of our algorithms and compares their
performances with MATLAB's and Chebfun's algorithms. We note cases where our
algorithms are successful in meeting the error tolerance, and other
algorithms are not.

\subsection{Related Work on Adaptive Algorithms}

Adaptive algorithms relieve the user of having to specify the number of samples
required. Only the desired error tolerance is needed. Existing adaptive
numerical algorithms for function approximation, such as the MATLAB toolbox
Chebfun \citep{TrefEtal16a}, succeed for some functions, but fail for others. No
theory explains for which $f$ Chebfun succeeds. A corresponding situation exists
for minimization algorithms, such as \texttt{min} in Chebfun or MATLAB's
built-in \texttt{fminbnd} \citep{Bre13a, For77}.

Our theoretically justified Algorithms $A$ and $M$ build upon the ideas used to
construct the adaptive algorithms in \cite{HicEtal14b, Din15a, HicEtal14a,
HicJim16a, Jia16a, JimHic16a,Ton14a}. In all those cases, a
cone, $\cc$, of input functions is identified for which the adaptive algorithms
succeed, just as is done here. However, unlike the algorithms in \cite{HicEtal14b, Din15a,
HicEtal14a,Ton14a}, the definition of $\cc$ here does not
depend on a weaker norm. Also, unlike
the globally adaptive approximation and optimization algorithms in
\cite{HicEtal14b,Ton14a}, the algorithms proposed here are locally adaptive,
sampling the interval $[a, b]$ nonuniformly.

Novak \cite{Nov96a} summarizes the settings under which adaption may provide an
advantage over nonadaption. For linear problems, such as~\eqref{appxprob},
adaption has no advantage if the set of functions being considered is symmetric
and convex \cite[Theorem 1]{Nov96a}, \cite[Chapter 4, Theorem
5.2.1]{TraWasWoz88}, \cite{Woz88a}.
The cone $\cc$ defined for our approximation problem
\eqref{appxprob} is symmetric, but \emph{not} convex. Plaskota et~al.~\cite{PlaEtal08a}
have developed adaptive algorithms for functions with singularities. Our
algorithms are not designed for such functions. Rather they are
designed to be efficient when the second derivative is large in a small part of
the domain.

Plaskota \cite{Pla15a} has developed an adaptive Simpson's algorithm for
approximating $\int_a^b f(x) \, \dif x$ assuming that the fourth derivative $f^{(4)}(x) \ge 0$ for all
$x \in [a,b]$. His algorithm relies on divided differences, like ours do. His
error is asymptotically proportional to $\norm[\frac14]{f^{(4)}}$, which is
analogous to the $\norm[\frac12]{f''}$ that appears in our analysis.  Horn~\cite{Hor06a}
has developed an optimization algorithm for Lipschitz continuous functions that does not
require knowledge of the Lipschitz constant.

There is a significant literature on theoretically justified algorithms based on
interval arithmetic \cite{MoKeCl09, Rum10a}, which are implemented in
INTLAB~\cite{Rum99a}. This approach assumes that functions have interval inputs
and outputs. We focus on the more common situation where functions have point
inputs and outputs.

\section{The Cone, $\cc$, of Functions of Interest} \label{sec:cone}

Linear splines \eqref{splinedef} are the foundation for adaptive algorithms $A$
and $M$. To bound the error of the linear spline in \eqref{appxerrbda}, our
algorithms construct data-based upper bounds on $\norm[{[\alpha, \beta]}]{f''}$
in terms of divided differences. For these bounds to hold, we must assume that
$f''(x)$ does not change drastically with respect to a small change in $x$.
These assumptions define our cone of functions, $\cc$, for which our algorithms
ultimately apply.

Let $p$ denote the quadratic Lagrange interpolating polynomial at the nodes
$\{\alpha, (\alpha + \beta)/2, \beta\}$, which may be written as
\begin{align}
\nonumber
   p(x) & := f(\alpha) + \frac{(x-\alpha)[f(\beta) - f(\alpha)]}{\beta - \alpha}  +
   (x-\alpha)(x-\beta) D(f,\alpha,\beta),
\\ D(f,\alpha,\beta) &:= \frac{2f(\beta) - 4f((\alpha + \beta)/2))
	+ 2f(\alpha)}{(\beta - \alpha)^2}. \label{divdiffdef}
\end{align}
For any $f \in \cw^{2,\infty}$, the function $f - p$ has at least three distinct
zeros on $[\alpha,\beta]$, so $f' - p'$ has at least two distinct zeros on
$(\alpha,\beta)$.  Specifically, there exist $\xi_\pm$ with $\alpha < \xi_- < (\alpha +
\beta)/2 < \xi_+ < \beta$ with $f'(\xi_\pm) - p'(\xi_{\pm}) = 0$. Thus,
\begin{align}
     \norm[-\infty,{[\alpha,\beta]}]{f''}
   & := \inf_{\alpha \le \eta < \zeta \le \beta} \abs{\frac{f'(\zeta) - f'(\eta)}{\zeta - \eta}}
   \nonumber
\\ & \le \abs{\frac{f'(\xi_+) - f'(\xi_-)}{\xi_+ - \xi_-}}
   = \abs{\frac{p'(\xi_+) - p'(\xi_-)}{\xi_+ - \xi_-}} = 2 \abs{D(f,\alpha,\beta)}  \nonumber
\\ & \le \sup_{\alpha \le \eta < \zeta \le \beta} \abs{\frac{f'(\zeta) - f'(\eta)}{\zeta - \eta}}
=: \norm[{[\alpha,\beta]}]{f''}. \label{NDDbdm}
\end{align}
This inequality tells us that twice the divided difference,
$2\abs{D(f,\alpha,\beta)}$, is a \emph{lower} bound for
$\norm[{[\alpha,\beta]}]{f''}$, which by itself is not helpful. But
$2\abs{D(f,\alpha,\beta)}$ is an \emph{upper} bound for
$\norm[-\infty,{[\alpha,\beta]}]{f''}$. The cone of interesting functions,
$\cc$, will contain those $f$ for which $\norm[{[\alpha,\beta]}]{f''}$ is not
drastically greater than the maximum of $\norm[-\infty,{[\beta - h_-,
\alpha]}]{f''}$ and $\norm[-\infty,{[\beta, \alpha+ h_+]}]{f''}$, where $h_{\pm} > \beta -
\alpha$.

The cone $\cc$ is defined in terms of two numbers: an integer $n_{\ninit} \ge 5$
and a  number $\fC_0 \ge 1$. Let
\begin{equation}
\label{hCdef}
\fh := \frac{3(b-a)}{n_{\ninit}-1}, \qquad \fC(h) := \frac{\fC_0 \fh}{\fh - h}
\mbox{\, for \,} 0<h<\fh.
\end{equation}
For any $[\alpha, \beta] \subset [a,b]$ and any $h_{\pm}$ satisfying $0 <
\beta - \alpha < h_{\pm} < \fh$, define
\begin{equation} \label{tbdef}
B(f'',\alpha,\beta,h_-,h_+) :=
\begin{cases}
   \max\bigl(\fC(h_{-}) \norm[-\infty,{[\beta-h_-,\alpha]}]{f''},\fC(h_{+})
   \norm[-\infty,{[\beta, \alpha+h_+]}]{f''}\bigr),
\\                                             \hspace{27ex} a \le \beta- h_- <  \alpha + h_+ \le b,
\\ \fC(h_{-}) \norm[-\infty,{[\beta-h_-,\alpha]}]{f''},  \quad a \le \beta- h_- < b< \alpha + h_+,
\\ \fC(h_{+}) \norm[-\infty,{[\beta, \alpha+h_+]}]{f''},
\quad \beta- h_- <  a  < \alpha + h_+ \le b.
\end{cases}
\end{equation}
\begin{multline} \label{conedef}
 \cc := \Bigl \{
 f  \in    \cw^{2,\infty}:   \norm[{[\alpha,\beta]}]{f''}  \le B(f'',\alpha,\beta,h_-,h_+)
 \text{ for all } [\alpha,\beta] \subset [a,b]
\\ \text{and } h_{\pm} \in (\beta - \alpha,  \fh)  \Bigr \}.
\end{multline}
The set $\cc$ is a cone because $f \in \cc \implies cf \in \cc$ for all real
$c$. The integer $n_{\ninit}$ is the initial number of subintervals in
Algorithms $A$ and $M$. \begin{FJHchange}The parameter $\fC_0$ is some number
no less
than one for
which
\[
\lim_{h \to 0} \norm[{[x-h,x+h]}]{f''} \le
\fC_0 \lim_{h \to 0} \norm[-\infty,{[x-h,x+h]}]{f''}, \qquad \forall x \in (a,b), \ f \in \cc.
\]\end{FJHchange}
Increasing either $n_{\ninit}$ or $\fC_0$ expands the cone to
include more functions.

Figure \ref{fig:ConeDef} depicts the second derivative of a typical function in
$\cw^{2,\infty}$.  In this figure
$\norm[-\infty,{[\beta-h_-,\alpha]}]{f''} = \abs{f''(\beta - h_-)} = 0$, which means that the
behavior of $f''$ to the left of $[\alpha,\beta]$ cannot help provide an upper bound on $
\norm[{[\alpha,\beta]}]{f''}$.  However, $\norm[-\infty,{[\beta,\alpha+h_+]}]{f''} =
\abs{f''(\alpha + h_+)} > 0$, so this $f$ may lie in the cone $\cc$ provided that
$\fC(h_+)$ is large enough.  The possibility of points in $[a,b]$ where~$f''$ vanishes
motivates the definition of $B(f'',\alpha,\beta,h_-,h_+)$ to depend on the behavior of $f''$
to both the left and right of $[\alpha,\beta]$.  One may note that if $f''$ vanishes at two
points that are close to each other or at a point that is close to either $a$ or~$b$, then
$f$ will lie outside $\cc$.  The definition of ``close'' depends on
$(b-a)/n_{\ninit}$.

\begin{figure}
	\centering
	\includegraphics[width = 8 cm]{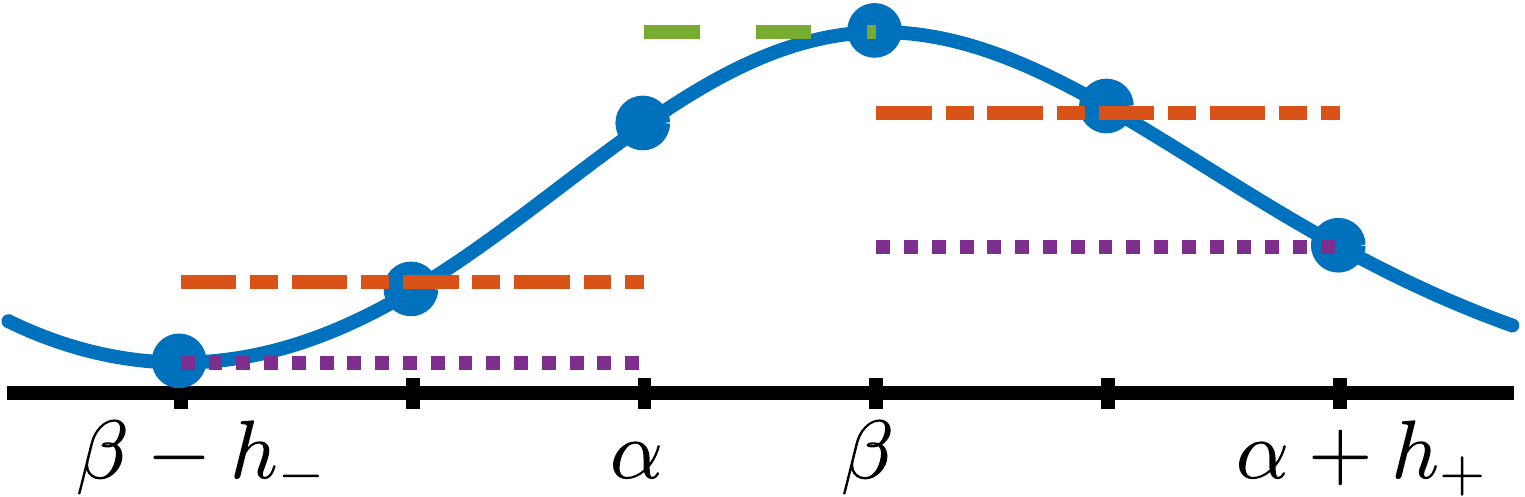}
	\caption{For some sample $f$, a plot of $\abs{f''(x)}$ (solid),
	$\norm[{[\alpha,\beta]}]{f''}$ (dashed),
   $\norm[-\infty,{[\beta - h_-,\alpha]}]{f''}$ and $\norm[-\infty,{[\beta,
    \alpha+h_+]}]{f''}$ (dotted), and
    $2\abs{D(f,\beta - h_-,\alpha)}$ and $2\abs{D(f,\beta,\alpha+h_+)}$ (dot-dashed).
    All figures in this paper are reproducible by
    \texttt{LocallyAdaptivePaperFigs.m} in GAIL~\cite{ChoEtal15a}.
    \label{fig:ConeDef}}
\end{figure}

We give an example of a family of functions whose members lie
inside $\cc$ if they are not too spiky.
Consider the following hump-shaped function defined on $[-1,1]$, whose
second derivative has jump discontinuities:
\begin{align} \label{f3def}
f_1(x) & = \begin{cases} \displaystyle
   \frac{1}{2\delta^2} \Bigl [4 \delta^2 + (x-c)^2 + (x-c-\delta)\abs{x-c-\delta}
\\ \qquad \qquad
    - (x-c+\delta)\abs{x-c+\delta} \Bigr ], & \abs{x-c} \le 2\delta,
\\ 0, & \text{otherwise},
\end{cases}
\\ \nonumber
f''_1(x) & =
\begin{cases} \displaystyle
    \frac{1}{\delta^2} [1 + \sign(x-c-\delta) - \sign(x-c+\delta)\begin{FJHchange}
    ]
    \end{FJHchange}, & \quad \abs{x-c} \le 2\delta,
\\ 0, & \quad \text{otherwise}.
\end{cases}
\end{align}
Here $c$ and $\delta$ are parameters satisfying $-1 \le c-2 \delta < c+ 2\delta
\le 1$. This function and its second derivative are shown in Figure~\ref{f3fig}(a)
for $-c=\delta = 0.2$.

\begin{figure}[tb]
\centering
\begin{tabular}{cc}
\includegraphics[width=5.7cm]{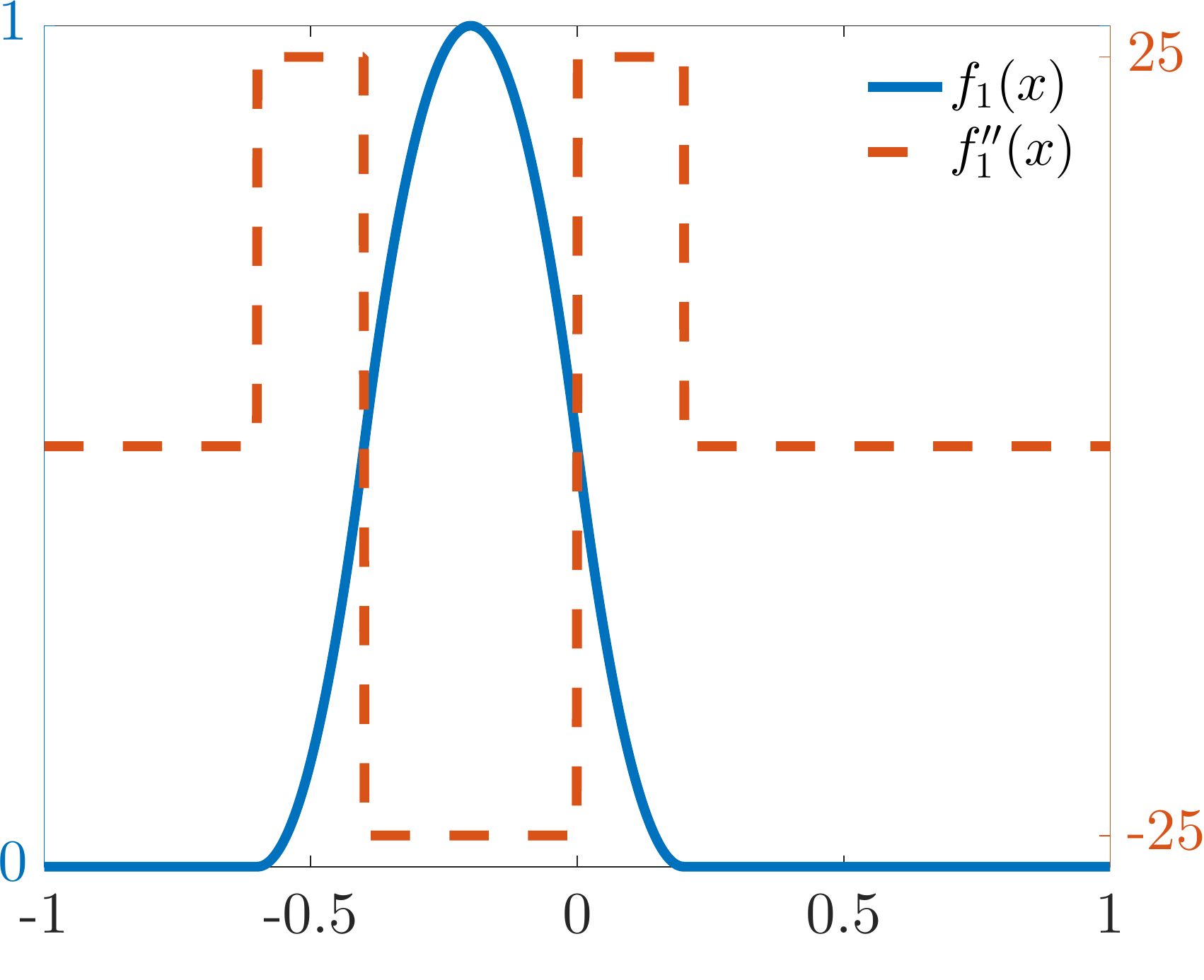}
& \includegraphics[width=5.7cm]{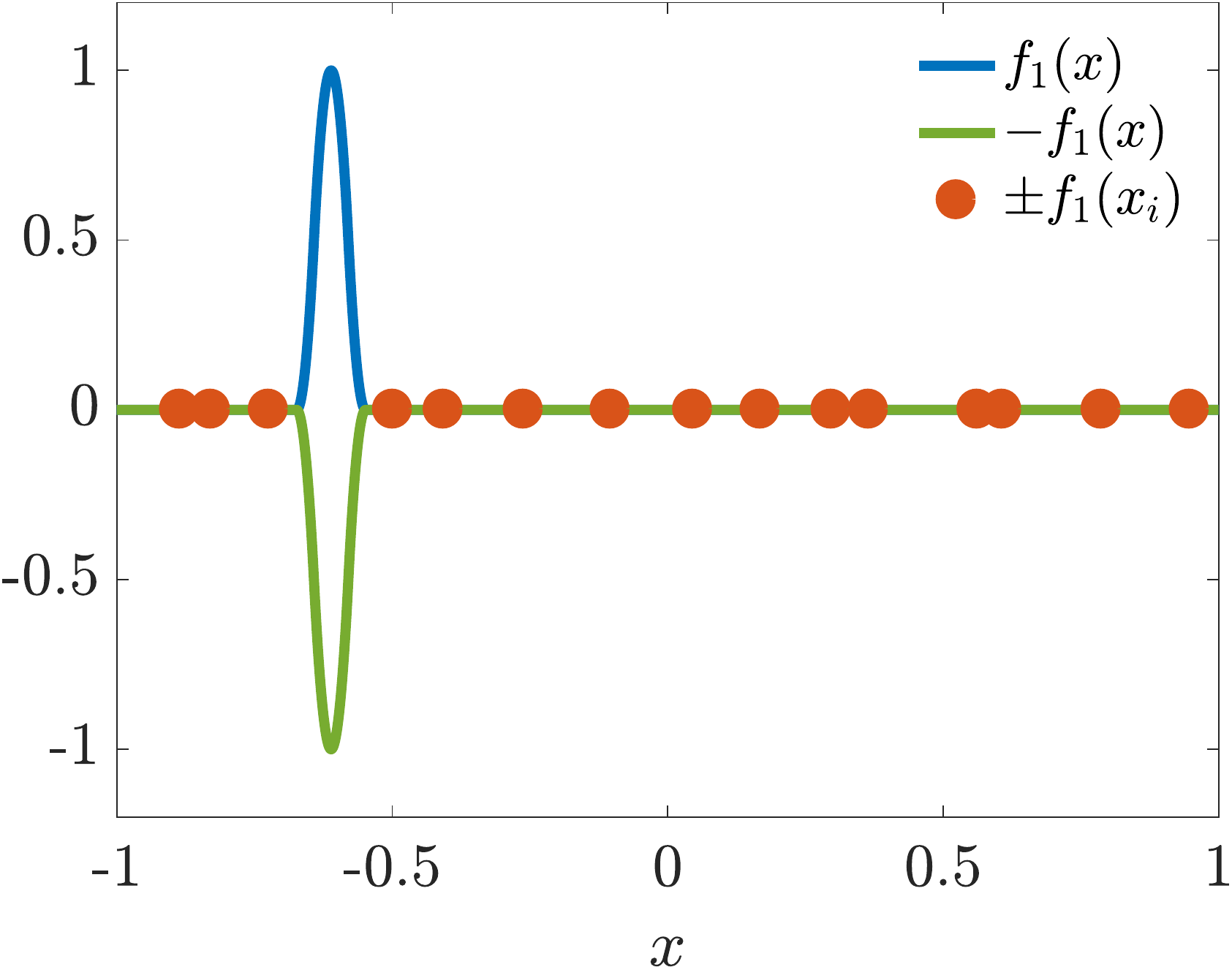}
\\[1ex] (a) & (b)
\end{tabular}
\caption{(a) The example $f_1$ with $-c=\delta = 0.2$ and its piecewise constant
second derivative.
(b) The fooling functions $\pm f_1$ used to prove~\eqref{lowbdW} (with different choices
of $c$ and $\delta$). The case
$n=15$ is shown. \label{f3fig}}
\end{figure}

If the hump is wide enough, i.e., $\delta \ge 2 \fh$, then $f_1 \in \cc$ for any
choice of $\fC_0 \ge 1$. For any $[\alpha,\beta]\subseteq [-1,1]$ and
$h_{\pm}$ satisfying the conditions in the definition of~$\cc$ in~\eqref{conedef},
it follows that
\begin{align*}
    & \norm[{[\alpha,\beta]}]{f''_1} = \frac{1}{ \delta^{2} } =
\begin{cases}
\norm[-\infty,{[\beta,\alpha+h_+]}]{f''_1}  
   & \text{if } \alpha\text{ or } \beta \in [c - 2\delta,  c - 1.5\delta]
\\ & \quad\quad \cup [c - \delta,  c - 0.5\delta] \cup [c + \delta, c+1.5\delta],
\\ \norm[-\infty,{[\beta - h_-,\alpha]}]{f''_1} 
   & \text{if } \alpha\text{ or } \beta \in [c - 1.5\delta,  c - \delta]
\\ & \quad\quad  \cup [c - 0.5 \delta,  c + \delta] \cup [c + 1.5\delta, c+2\delta].
\end{cases}
\end{align*}
Thus, $B(f'',\alpha,\beta,h_-,h_+)\ge\norm[{[\alpha,\beta]}]{f''_1}$ for $\beta
\ge c - 2\delta$ or $\alpha \le c + 2\delta$.
For $[\alpha,\beta] \subset [-1,c
- 2\delta) \cup (c+2\delta, 1]$, it follows that $\norm[{[\alpha,\beta]}]{f''_1} =
0$, so $B(f'',\alpha,\beta,h_-,h_+)\ge\norm[{[\alpha,\beta]}]{f''_1}$
automatically. Thus, the definition of the cone is satisfied.

However, if the hump is too narrow, i.e., $\delta < 2 \fh$, the function $f_1$
is too spiky to lie in the cone $\cc$ regardless of how $\fC_0$ is defined. For
$\alpha$, $\beta$, and $h$ satisfying
\[
0 \ <  \ c - 1.5\delta -\alpha  \ = \  \beta - c+ 1.5\delta \ <  \ 0.5\delta \  <  \ c-1.5\delta- \beta + h \ < \ \fh,
\]
it follows that
\begin{gather*}
   \beta -  h  \ < \ c - 2 \delta \ < \ \alpha \ < \ c - 1.5\alpha  \ < \ \beta \ < \ c - \delta \ < \ \alpha + h,
\\ \norm[-\infty,{[\beta - h,\alpha]}]{f''_1} \   =  \ \norm[-\infty,{[\beta,\alpha+h]}]{f''_1} \ = \ 0
\ < \ \delta^{-2} \ = \ \norm[{[\alpha,\beta]}]{f''_1}.
\end{gather*}
This violates the definition of $\cc$. This example illustrates how the choice
of $n_{\ninit}$, or equivalently $\fh$, influences the width of a spiky function and
determines whether it lies in $\cc$.

\section{The Function Approximation Algorithm, $A$}\label{sec:fappx}

\subsection{Algorithm $A$} \label{subsec:appxalgo}

The idea of Algorithm $A$ is to use divided differences to provide upper bounds
on $\norm[-\infty,{[\beta - h_-, \alpha]}]{f}$ and $\norm[-\infty{[\beta,
\alpha+h_+]}]{f}$ via \eqref{NDDbdm}, which then provide an upper bound on
$\norm[{[\alpha,\beta]}]{f}$ via the definition of the cone, $\cc$, in
\eqref{conedef}. This in turn yields an upper bound on the spline error via
\eqref{appxerrbda}. After stating the algorithm, its effectiveness is proven.

\begin{algoA} \label{AlgoA}
For some finite interval $[a,b]$, integer $n_{\ninit}\ge 5$, and constant $\fC_0 \ge
1$, let $\fh$ and $\fC(h)$ be defined as in~\eqref{hCdef}. Let $f:[a,b] \to
\reals$ and $\abstol >0$ be user inputs. Define the number of subintervals,
$n=n_{\ninit}$, and the iteration number, $l = 0$. Define the initial partition
of equally spaced points, $x_{0:n}$, and an index set of subintervals:
\[
h_0 =\frac{b-a}{n}, \qquad  x_i=a+ ih_0,  \qquad\ i \in \zton, \qquad \mathcal{I} = 1\!:\!(n-1).
\]

\begin{enumerate}[\em Step 1.]

\item \label{stage1} \emph{Check for convergence.} For all $i \in \ci$ compute
\begin{equation} \label{oerrdef}
\oerr_i = \frac{1}{8} \fC(3h_l)\abs{f(x_{i+1}) - 2f(x_i) + f(x_{i-1})}.
\end{equation}
Let $\widetilde{\mathcal{I}} = \left\{i \in \mathcal{I}: \oerr_i  > \abstol \right\}$
be the index set for those $\oerr_i $ that are too large.   If $\widetilde{\mathcal{I}} =
\emptyset$, return the linear spline $A(f,\abstol) = S(f, x_{0:n})$ and terminate
the algorithm. Otherwise, continue to the next step.

\item \label{stage2} \emph{Split the subintervals as needed.}
Update the present partition, $x_{0:n}$, to include the subinterval midpoints
\[
  \frac{x_{i-2} +x_{i-1}}{2}, \ \frac{x_{i-1} +x_{i}}{2},
\ \frac{x_{i} +x_{i+1}}{2}, \  \frac{x_{i+1} +x_{i+2}}{2}, \qquad i \in \widetilde{\ci}.
\]
(The leftmost midpoint is only needed for $i \ge 2$, and the rightmost midpoint
is only needed for $i \le n-2$.) Update the set $\ci$ to consist of the new
indices corresponding to the old points
\[
x_{i-1}, \ \frac{x_{i-1} +x_{i}}{2}, \ \frac{x_{i} +x_{i+1}}{2}, \  x_{i+1},
\qquad i \in \widetilde{\ci}.
\]
(The point $x_{i-1}$ is only included for $ i \ge 2$, and the point $x_{i+1}$ is
only included for $i \le n-2$.) Let $l \leftarrow l+1$ and $h_l = h_{l-1}/2$.  Return to
Step~\ref{stage1}.
\end{enumerate}
\end{algoA}

\begin{theorem} \label{thm:algAworks}
	Algorithm $A$ defined above satisfies~\eqref{appxprob} for functions in the
	cone $\cc$ defined in~\eqref{conedef}.
\end{theorem}

\begin{proof}
For every iteration $l$ and every $i \in \ci$, the definitions in this algorithm
imply that $x_i-x_{i-1} = x_{i+1} - x_i = h_l = 2^{-l}h_0$, and
	\begin{align}
	\oerr_i
	& =  \frac{1}{4} \fC(3h_l) h_l^2 \abs{D(f,x_{i-1},x_{i+1})}
	\qquad \text{by \eqref{divdiffdef}}  \label{divdiffbd}\\
	& \ge \frac{1} 8 \fC(3h_l) h_l^2 \norm[-\infty,{[x_{i-1},x_{i+1}]}]{f''}
	\qquad \text{by \eqref{NDDbdm}}. \label{oerrlobd}
	\end{align}
We show that when  all $\oerr_i$ get small enough, Algorithm $A$ terminates
successfully.

For all $x \in [a,b]$, let $\Ixl$ be the closed interval with width $h_l$
containing $x$ that \emph{might} arise at some stage in Algorithm $A$ as $[x_{i_l-1},
x_{i_l}]$ for some $i_l \in \oton$. (The dependence of $n$ on $l$ is
suppressed.) Specifically this interval is defined for all $ x \in [a,b]$ and $l
\in \mathbb{N}_0$ as
	\begin{equation}\label{Ixldef}
	\Ixl := \left[a+jh_l,a+(j+1) h_l\right], \ \ j=\min \left(\left\lfloor\frac{(x-a)}{h_l}\right\rfloor,
	2^l n_{\ninit}-1 \right ).
	\end{equation}
Let $\ell(x)$ be defined such that $I_{x,\ell(x)}$ is the final subinterval in
Algorithm $A$ that contains $x$ when the algorithm terminates. We need to
establish that $\norm[I_{x,\ell(x)}]{f - S(f)} \le \abstol$ for every $x \in
[a,b]$.

Fix $x \in [a+\fh, b-\fh]$. The proof for $x \in [a,a+\fh) \cup (b-\fh,b]$ is
similar. By \eqref{oerrlobd} there exists some $l_- \le \ell(x)$ for which $I_{x,{l_-}}
=[x_{i_{l_-}-1},x_{i_{l_-}}]$ and
\begin{subequations} \label{datalesstol}
\begin{equation}
\frac{1} 8 \fC(3h_{l_-}) h_{l_-}^2 \norm[-\infty,{[x_{i_{l_-}-3},x_{i_{l_-}-1}]}]{f} \le \oerr_{i_{l_-}-2}
\le \abstol .
\end{equation}
There also exists an $l_+ \le \ell(x)$ such that $I_{x,l_+} =[x_{i_{l_+}-1},x_{i_{l_+}}]$ and
\begin{equation}
\frac{1} 8 \fC(3h_{l_+}) h_{l_+}^2 \norm[-\infty,{[x_{i_{l_+}},x_{i_{l_+}+2}]}]{f}
\le \oerr_{i_{l_+}+1} \le \abstol .
\end{equation}
\end{subequations}
Noting that $x_{i_{l_{\pm}}-1} \le x_{i_{\ell(x) -1}} < x_{i_{\ell(x)}} \le x_{i_{l_{\pm}}} $, we may
conclude that
\begin{align*}
   & \norm[I_{x,\ell(x)}]{f - S(f)} \le \frac{1}{8}h_{\ell(x)}^2\norm[I_{x,\ell(x)}]{f''} \qquad
   \text{by \eqref{appxerrbda}}
\\ & \le \frac{1}{8} h_{\ell(x)}^2 B(f,x_{i_{\ell(x) - 1}},
x_{i_{\ell(x)}},x_{i_{\ell(x)}}-x_{i_{l_-}-3},x_{i_{l_+}+2} - x_{i_{\ell(x) - 1}}) \quad
 \text{by  \eqref{conedef}}
\\ & \le \frac{1}{8} h_{\ell(x)}^2 \max \Bigl(\fC(x_{i_{\ell(x)}}-x_{i_{l_-}-3})
\norm[-\infty,{[x_{i_{l_-}-3},x_{i_{\ell(x) - 1}}]}] {f''}, \\
&\qquad  \fC(x_{i_{l_+}+2} - x_{i_{\ell(x)}-1})
\norm[-\infty,{[x_{i_{\ell(x)}}, x_{i_{l_+}+2}]}] {f''}  \Bigr) \quad
\text{by the definition of $B$ in \eqref{tbdef}}\\
& \le  \max \Bigl( \frac{ h_{l_-}^2}{8}\fC(3h_{l_-})
\norm[-\infty,{[x_{i_{l_-}-3},x_{i_{l_-}- 1}]}] {f''},
   \ \frac{h_{l_+}^2}{8}\fC(3h_{l_+})   \norm[-\infty,{[x_{i_{l_+}}, x_{i_{l_+}+2}]}] {f''}  \Bigr) \\
& \qquad \qquad  \text{because $h_{\ell(x)} \le h_{l_{\pm}}$ and $\fC$ is
	non-decreasing}\\
& \le \abstol \qquad \text{by \eqref{datalesstol}}.
\end{align*}
This concludes the proof.\end{proof}

Figure \ref{f3foolplot}(a) displays the function $-f_1$ defined in
\eqref{f3def} for a certain choice of parameters, along with the data used to
compute the linear spline approximation $A(-f_1,0.02)$ by the algorithm
described above. Note that $-f_1$ is sampled less densely where it is flat.

\begin{figure}[t]
	\begin{tabular}{cc}	
		\includegraphics[width=5.7cm]{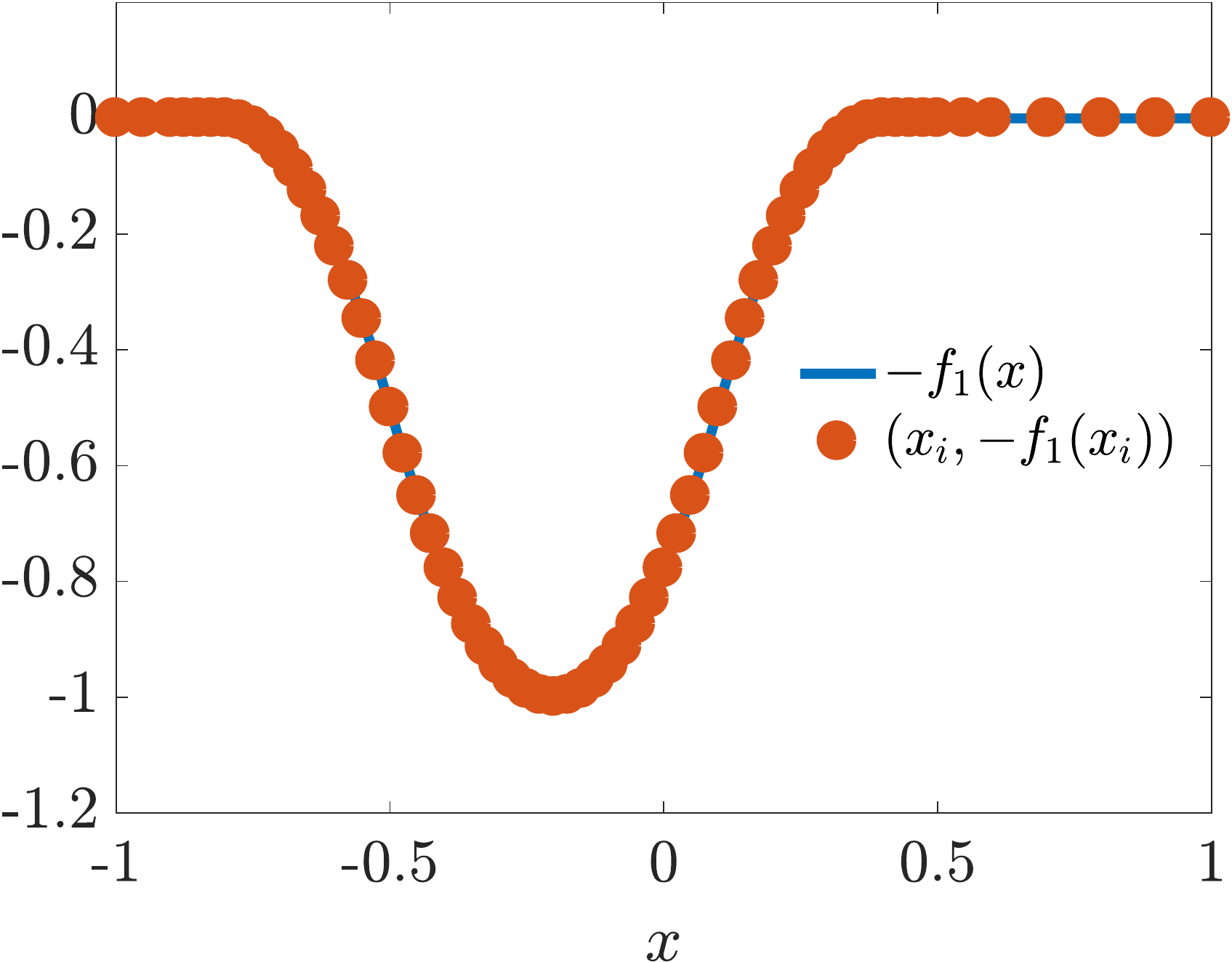}
		&\includegraphics[width=5.7cm]{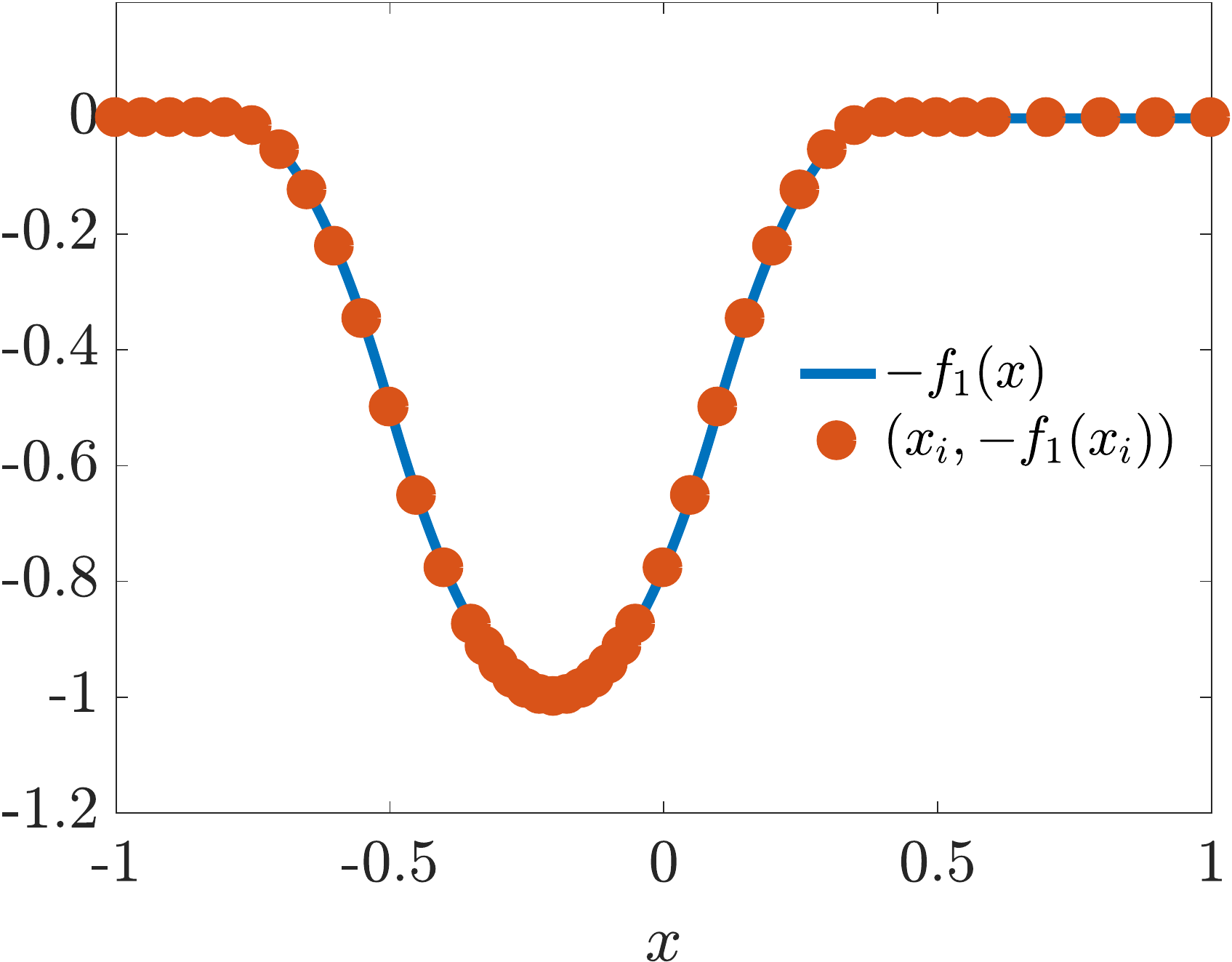}
		\\ (a) & (b)
	\end{tabular}
	\caption{(a) The nonuniform sampling density of Algorithm $A$ for input function
		$-f_1$ defined by $\delta = 0.3$ and $c = -0.2$. A total of~$3$ iterations
		and~$65$ points are used to meet the error tolerance
		of~$0.02$.  We have chosen $n_{\ninit} = 20$ and $\fC_0 = 10$.
		(b) The same situation as in (a), but now
		with Algorithm $M$. Still~$3$ iterations but only~$43$ nonuniform
		sampling points are needed to obtain the minimum of $-f_1$. 	
		\label{f3foolplot}}
\end{figure}

\subsection{The Computational Cost of $A$} \label{subsec:appxcost}

In this section, we investigate the computational cost of our locally adaptive
algorithm. Recall the definitions of $h_l$, $\Ixl$, and $\ell(x)$ from the
previous subsection. Let
$\bar{I}_{x,l}$ be a similar interval with generally five times the
width of~$\Ixl$:
\begin{equation} \label{barIxldef}
\bar{I}_{x,l}=\left[a+\max(0,j-3)h_l, a+ \min(j+2,2^l n_{\ninit})h_l\right] \supset \Ixl,
\end{equation}
with the same $j$ as in \eqref{Ixldef} above.  Let
\begin{equation}\label{eqn:defoflx}
L(x) = \min \left\{ l \in \mathbb{N}_0 :  \frac{1}{8} \fC\left(3h_l\right)h_l^2\norm[\bar{I}_{x,l}]{f''} \le \abstol \right\}.
\end{equation}
Note that $L(x)$ does depend on $f$ and $\abstol$, although this dependence is
suppressed in the notation.

We now show that $\ell(x) \le L(x)$. At each iteration of Algorithm $A$, $x$
lies in $\Ixl$ for some $l$, and by the time Algorithm $A$ terminates, all
values of $l = 0, \ldots, \ell(x)$ are realized. If $\ell(x) > L(x)$, then at
 iteration $L(x)$, the interval  $I_{x,L(x)}$ must be split in Step~\ref{stage2} of $A$. So,
$I_{x,L(x)}$ has width $h_{L(x)}$ and corresponds to $[x_{i-1},x_i]$ for some
$i$. We assume that $i \in 3\!:\!n-2$; the other cases have a similar proof.
According to Step~\ref{stage2} of Algorithm $A$, the only way for
$[x_{i-1},x_i]$ to be split is if $\oerr_{i-2}$, $\oerr_{i-1}$, $\oerr_{i}$, or
$\oerr_{i+1}$ is larger than $\abstol$. However, in the proof of
Theorem~\ref{thm:algAworks} it is noted that for $k \in \{-2, -1, 0, 1\}$,
\begin{align}
\nonumber
\oerr_{i+k} & = \frac{1}4 \fC(3h_{L(x)}) h_{L(x)}^2 \abs{D(f,x_{i-1+k},x_{i+1+k})} \qquad
\text{by \eqref{divdiffbd}} \\
\nonumber
& \le \frac{1} 8  \fC(3h_{L(x)}) h_{L(x)}^2\norm[{[x_{i-1+k},x_{i+1+k}]}]{f''}  \qquad
\text{by \eqref{NDDbdm}}\\
& \le \frac{1} 8  \fC(3h_{L(x)}) h_{L(x)}^2 \norm[\bar{I}_{x,L(x)}]{f''} \le \varepsilon \qquad
\text{by \eqref{barIxldef} and \eqref{eqn:defoflx}}. \label{oerrbd}
\end{align}
This is a contradiction, so in fact, $\ell(x) \le L(x)$, which is used to prove an upper
bound on the
computational cost of Algorithm $A$.

\begin{theorem}\label{thm:cost}
Let $\cost(A,f,\abstol)$ denote the number of functional evaluations required by
$A(f,\abstol)$. This computational cost has the following upper bound:
\begin{equation*}
\cost(A,f,\abstol) \le \frac{1}{h_0}\int_a^b 2^{L(x)} \, \dif x +1
= \int_a^b \frac{1}{h_{L(x)}} \, \dif x +1, \\
\end{equation*}
where $L(x)$ is defined in~\eqref{eqn:defoflx}.
\end{theorem}

\begin{proof}
Let $x_{0:n}$ be the final partition when $A(f,\abstol)$
successfully terminates. Note that $2^{\ell(x)}$ is constant for $x \in
I_{x_{i-1},\ell(x_{i-1})} = [x_{i-1},x_{i}]$ for $i \in \oton$. Furthermore
$\int_{x_{i-1}}^{x_{i}} 2^{\ell(x)} \, \dif x = h_0$. Then the number of
function values required is
\begin{equation*}
n+1 = 1 + \sum_{i=1}^{n} 1 = 1 + \sum_{i=1}^{n} \frac{1}{h_0}
\int_{x_{i-1}}^{x_{i}} 2^{\ell(x)} \, \dif  x = 1 + \frac{1}{h_0}\int_a^b 2^{\ell(x)} \, \dif x.
\end{equation*}
Noting that $\ell(x) \le L(x)$ establishes the formula for $\cost(A,f,\abstol)$.
\end{proof}

From the definition of $L(x)$ in~\eqref{eqn:defoflx}, we know that
\begin{align*}
\frac{1}{h_{L(x)}} = \frac{2}{h_{L(x)-1}}
& < 2 \sqrt{\frac{\fC\left(3 h_{L(x)-1} \right)\norm[\bar{I}_{x,L(x)-1}]{f''} }{8\abstol}}
= \sqrt{\frac{\fC\left(6 h_{L(x)} \right)\norm[\bar{I}_{x,L(x)-1}]{f''} }{2\abstol}}.
\end{align*}
As $\abstol \to 0$, $L(x) \to \infty$, $h_{L(x)} \to 0$, and  $\norm[\bar{I}_{x,L(x)-1}]{f''}$
approaches $|f''(x)|$.  Thus, the small $\abstol$ asymptotic upper bound on
computational cost is
\begin{align*}
\cost(A,f,\abstol)
&\lesssim \int_a^b \sqrt{\frac{\fC\left(0\right)   |f''(x)|}{2\abstol}} \, \dif x +1
=\sqrt{\frac{\fC_0 \norm[\frac 12]{f''}}{2\abstol}} +1  \\
& \le (b-a)\sqrt{\frac{\fC_0 \norm{f''}}{2\abstol}} +1  \qquad
\text{by \eqref{halflessinf} below}.
\end{align*}

For functions in the cone $\cc$, the (quasi-)seminorms $\norm{f''}$ and $\norm[\frac
12]{f''}$ are equivalent, but  for functions in $ \cw^{2,\infty}$ they are not, as shown in the following proposition.

\begin{prop} \label{equivnormprop}
	The quantities $\norm{f''}$ and $\norm[\frac 12]{f''}$ bound each other as follows:
	\begin{subequations}
		\begin{gather}
		\label{halflessinf}
		(b-a)^2\norm[-\infty,{[\alpha,\beta]}]{f''}  \le \norm[\frac 12,{[\alpha,\beta]}]{f''}
		\le (b-a)^2 \norm[{[\alpha,\beta]}]{f''} \quad \forall f \in \cw^{2,\infty}, \\
		\label{inflesshalf}
		\frac{4\fh^2}{27\fC_0} \norm{f''} \le \norm[\frac 12]{f''} \qquad \forall f \in \cc, \\
		\label{infhalfnoteq}
		\sup_{f \in \cw^{2,\infty} :  \norm[\frac 12]{f''} \le 1}  \norm{f''}
		= \infty.
		\end{gather}
	\end{subequations}
\end{prop}
\begin{proof}
The first inequality follows from the definitions of the (quasi-)norms:
\begin{multline} \label{onebdm}
 (\beta - \alpha)^2 \norm[-\infty,{[\alpha,\beta]}]{f''}
 = \biggl\{\sqrt{\norm[-\infty,{[\alpha,\beta]}]{f''}}  \int_{\alpha}^{\beta} \, \dif x \biggr\}^2
 \le \biggl\{ \int_{\alpha}^{\beta} \sqrt{\abs{f''(x)}} \, \dif x \biggr\}^2\\
 = \norm[\frac 12,{[\alpha,\beta]}]{f''}
 \le\biggl\{\sqrt{\norm[{[\alpha,\beta]}]{f''}}  \int_{\alpha}^{\beta} \, \dif x \biggr\}^2
 \le  (\beta - \alpha)^2 \norm[{[\alpha,\beta]}]{f''}.
\end{multline}

The second inequality comes from the cone definition.  Since
$\norm{f''}  = \norm[{[\alpha,\beta]}]{f''}$  for some interval $[\alpha,\beta]$ whose
width can be made arbitrarily small, we have
\begin{align*}
\norm[{[\alpha,\beta]}]{f''}
& \le \inf\bigl\{B(f,\alpha,\beta,h,h) : h \in (\beta - \alpha,\fh) \bigr \} \quad \text{by}~\eqref{conedef} \\
& \le  \inf\bigl\{\fC(h) \max\bigl(\norm[-\infty,{[\beta - h, \alpha]}]{f''},
\norm[-\infty,{[\beta, \alpha+h]}]{f''}\bigr) : h \in (\beta - \alpha,\fh) \bigr \} \\
& \le \inf_{\beta - \alpha < h < \fh} \frac{\fC(h)}{(h-\beta+\alpha)^2}\norm[\frac 12]{f''} \\
& \qquad \qquad\text{by}~\eqref{onebdm}, \text{ and since }
\norm[{\frac12,[\alpha,\beta]}]{f''} \le \norm[\frac 12]{f''} \ \forall [\alpha,\beta]
\subseteq [a,b] \\
& \le \inf_{0 < h < \fh} \frac{\fC(h)}{h^2}\norm[\frac 12]{f''} \qquad
\text{since $\beta-\alpha$ may be made arbitrarily small}\\
& = \frac{27\fC_0}{4\fh^2}\norm[\frac 12]{f''} \qquad \text{by \eqref{hCdef}}.
\end{align*}
When $\fh$ is small, it is possible for $\norm[\frac 12]{f''} $ to be quite
small in comparison to $\norm{f''}$. This occurs when $f''$ is rather spiky.

The hump function $f_1$ in~\eqref{f3def}  satisfies
$
 {\norm{f_1''}}/{\norm[\frac 12]{f_1''}}  =  { \delta^{-2}}/{16}.
$
By making~$\delta$ small enough, we may make this ratio arbitrarily large, thus proving
\eqref{infhalfnoteq}. However, since $f_1
\notin \cc$ for $\delta < 2\fh$, this does not violate \eqref{inflesshalf}.
\end{proof}

\subsection{Lower Complexity Bound} \label{subsec:appxcomp}

The upper bound on the computational cost of Algorithm~$A$ provides an upper
bound on the complexity of problem~\eqref{appxprob}. We now construct lower
bounds on the complexity of the problem, i.e., the computational cost of the
best algorithm. We then observe that these lower bounds have the same asymptotic
behavior as the computational cost of Algorithm $A$.  \begin{FJHchange}
Our lower complexity
bounds
are
derived for subsets of functions in the balls, $\cb^{2,p}_{\sigma} = \{f \in W^{1,\infty} :
\norm[p]{f''} \le \sigma\}$, for $p = 1/2, \infty$.

\begin{theorem} 	\label{thm:A_complexity}
	Let $\sigma$ be any positive number, and $\cc$ be defined as in \eqref{conedef}.
	
	\begin{enumerate}
		\renewcommand{\labelenumi}{\roman{enumi}.}
	
	\item If $A^*$ solves~\eqref{appxprob} for all $f \in \cb^{2,\frac
		12}_{\sigma}$ and all $0 < \abstol < \sigma/16$, then
		\begin{subequations} \label{lowbdW}
		\begin{equation}
		 \cost(A^*,f,\abstol)
		= \infty.  \label{lowbdWhalf}
		\end{equation}
		\item If $A^*$ solves~\eqref{appxprob} for all $f \in \cb^{2,\infty}_{\sigma}$ and all
		$\abstol > 0$, then	
		\begin{equation}
	    \cost(A^*,f,\abstol)
		\ge \frac{(b-a)}{4} \sqrt{\frac{\sigma}{\abstol}} - 1.  \label{lowbdWinf}
		\end{equation}
		\end{subequations}

		\item If $A^*$ satisfies~\eqref{appxprob} for all $f \in \cc \cap \cb^{2,\frac
	12}_{\sigma}$ and all $\abstol > 0$, then
\begin{subequations} \label{lowbdC}
	\begin{equation}
	\cost(A^*,f,\abstol) \ge
	\sqrt{\frac{(\fC_0-1)\sigma}{16(\fC_0+1)\abstol}} -1. \label{lowbdChalf}
	\end{equation}
	\item If $A^*$ satisfies~\eqref{appxprob} for all $f \in \cc \cap \cb^{2,\infty}_{\sigma}$
	and all $\abstol > 0$, then 	
	\begin{equation}
	\cost(A^*,f,\abstol) \ge
	(b-a)\sqrt{\frac{(\fC_0-1)\sigma}{16(\fC_0+1)\abstol}} -1 \label{lowbdCinf}.
	\end{equation}
\end{subequations}

\end{enumerate}			

\end{theorem}

Note by comparing \eqref{lowbdWhalf} and \eqref{lowbdChalf} that the lower complexity
bound is significantly altered by restricting the
set of input functions from the whole ball of $ \cb^{2,\frac 12}_{\sigma}$ to the
intersection of that ball with the cone $\cc$. Also note that the lower bounds above
assume that the radius of the ball, $\sigma$, is known a priori, whereas for our Algorithm
$A$, \emph{no} bound on a norm of $f''$ is provided as  input. However, the
computational cost of Algorithm $A$ is asymptotically
the same as the computational cost of the best possible algorithm, $A^*$ in
\eqref{lowbdC}, as $\abstol \to 0$.
\end{FJHchange}

\begin{proof}
	The lower bounds are proved by constructing fooling functions for which
	Algorithm $A$ succeeds, and then showing that at least a certain number of
	samples must be used. The proofs of~\eqref{lowbdW} are simpler, so
	we start with them.
		
	Let $A^*$ be a successful algorithm for all $f \in \cw^{2,\infty}$, and consider
	the partition $x_{0:n+1}$, where $x_{1:n}$ are the data sites
	used to compute $A^*(0,\abstol)$.  We now allow the possibility of $a = x_0=x_1$
	and $x_n = x_{n+1} = b$. Choose any $j=1, \ldots, n+1$ with
	$x_j-x_{j-1} \ge (b-a)/(n+1)$. Let $f_1$ be defined as in~\eqref{f3def} with $c
	= (x_j+x_{j-1})/2$, and $\delta = (b-a)/[4(n+1)]$.
	
	For any real $\gamma$, it follows that $\gamma f_1(x_i)=0$ for $i=0, \ldots,
	n+1$. Figure \ref{f3fig}(b) illustrates this situation. Since $0$ and $\pm
	\gamma f_1$ share the same values at the data sites, then they must share the
	same approximation: $A^*(\pm \gamma f_1,\abstol) = A^*(0,\abstol)$. Moreover,
	$\cost(A^*,0,\abstol) = \cost(A^*,\pm \gamma f_1,\abstol) = n$. Since the
	approximations of~$0, -\gamma f_1$, and $\gamma f_1$ are identical, this implies that
	$\gamma$
	must be no greater than $\abstol$:
	\begin{align*}
	\abstol  &\ge \max(\norm{\gamma f_1 - A^*(\gamma f_1,\abstol)},
	\norm{-\gamma f_1 - A^*(-\gamma f_1,\abstol)}) \\
	&= \max(\norm{\gamma f_1 - A^*(0,\abstol)}, \norm{-\gamma f_1 - A^*(0,\abstol)}) \\
	& \ge \frac{1}{2} [\norm{\gamma f_1 - A^*(0,\abstol)}
	+ \norm{-\gamma f_1 - A^*(0,\abstol)}] \\
	& \ge \frac{1}{2} \norm{\gamma f_1 - (-\gamma f_1)} =  \norm{\gamma f_1}
	= \gamma 
	= \begin{cases} \displaystyle { \norm[\frac 12]{\gamma f_1''}}/{16} , \\
	\displaystyle \delta^2 	\norm{\gamma f_1''}
	=  \frac{(b-a)^2 \norm{\gamma f_1''}}{16(n+1)^2},
	\end{cases}
	\end{align*}
	since $\norm{f_1''} =  \delta^{-2}$,  and $\norm[\frac 12]{f_1''} = 16$. The top inequality
	cannot be satisfied unless $\sigma = \norm[\frac 12]{\gamma
	f_1''}$ is small enough, which establishes~\eqref{lowbdWhalf}. Solving the
	bottom inequality for $n$ in terms of $\sigma = \norm{\gamma f_1''}$
	establishes~\eqref{lowbdWinf}.

	Now, we prove the lower complexity bounds~\eqref{lowbdC}, assuming that $A^*$ is
	a successful algorithm for all $f \in \cc$. Let $f_0$ be defined as follows
	\begin{equation*}
	f_0(x) =\frac{x^2}{2}, \quad f_0''(x) = 1, \quad x \in [a,b]; \qquad \norm[\frac12]{f_0''}
	= (b -a)^2,  \quad \norm{f_0''} = 1.
	\end{equation*}
	Since $f_0''$ is constant, it follows that $f_0 \in \cc$,  and $A^*$ successfully
	approximates $\gamma f_0$ for any  $\gamma\ge0$.
	
	Consider the partition $x_{0:n+1}$, where $x_{1:n}$ are the data sites
	used to compute $A^*(\gamma f_0,\abstol)$, and we again allow the possibility of $a
	= x_0=x_1$
	and $x_n = x_{n+1} = b$.  Again choose any $j=1, \ldots, n+1$ with $x_j-x_{j-1}
	\ge (b-a)/(n+1)$, and let~$f_1$ be defined as in~\eqref{f3def} with $c =
	(x_j+x_{j-1})/2$, and $\delta = (b-a)/[4(n+1)]$. We construct two fooling
	functions:
	\begin{gather*}
	f_{\pm} = f_0 \pm \tgamma f_1, \qquad \tgamma =\frac{\fC_0-1}{\fC_0+1}\delta^2,
	\qquad
	\norm{f''_{\pm}} = 1+\frac{\tgamma}{\delta^2} = \frac{2\fC_0}{\fC_0+1} , \\
\norm[-\infty,{[\alpha,\beta]}]{f''_{\pm}} \ge 1-\frac{\tgamma}{\delta^2} = \frac{2}{\fC_0+1}
= \frac{\norm{f''_{\pm}} }{\fC_0} \qquad \forall [\alpha,\beta] \subseteq [a,b].
	\end{gather*}
	The above calculations show that $\gamma f_{\pm} \in \cc$ for all real
	$\gamma$. Moreover, the definition of $f_{\pm}$ ensures that $A^*(\gamma f_0) =
	A^*(\gamma f_{\pm})$, and $\cost(A^*,\gamma f_0) = \cost(A^*,\gamma f_{\pm}) =
	n$.
	
Analogously to the argument above, we show that $\gamma\tgamma$ must be no larger
than~$\abstol$:
	\begin{align*}
	\abstol & \ge \max(\norm{\gamma f_+-A(\gamma f_+,\abstol)},
	\norm{\gamma f_--A(\gamma f_-, \abstol)}) \\
	& \ge \frac{1}{2} \left[ \norm{\gamma f_+-A(\gamma f_+,\abstol)}
	+ \norm{\gamma f_--A(\gamma f_-, \abstol)} \right] \\
	& = \frac{1}{2} \left[ \norm{\gamma f_+-A(\gamma f_0,\abstol)}
	+ \norm{\gamma f_--A(\gamma f_0, \abstol)} \right] \\
	& \ge \frac{1}{2}  \norm{\gamma f_+ - \gamma f_-} =  \norm{\gamma \tgamma f_1}
	= \gamma \tgamma\\
	& = \begin{Bmatrix} \displaystyle
	\frac{\norm[\frac 12]{\gamma f_0''}}{(b-a)^2}, \\
	\norm{\gamma f_0''}
	\end{Bmatrix}  \cdot \frac{\fC_0-1}{\fC_0+1} \delta^2
	 = \begin{Bmatrix} \displaystyle
	\norm[\frac 12]{\gamma f_0''}, \\
	(b-a)^2\norm{\gamma f_0''}
	\end{Bmatrix}  \cdot \frac{\fC_0-1}{16(\fC_0+1)(n+1)^2}
	\end{align*}
	Substituting $\norm[\frac12]{\gamma f_0''} = \sigma$ in the top inequality and
	$\norm{\gamma f_0''} = \sigma$ in the bottom inequality, and then solving for
	$n$ yield the two bounds in~\eqref{lowbdC}.
\end{proof}

\section{The Minimization Algorithm, $M$} \label{sec:funmin}

\subsection{Algorithm $M$}  \label{sec:minalgo}

Our minimization algorithm $M$ relies on the derivations in the previous
sections. The main departure from Algorithm $A$ is the stopping criterion. It is
unnecessary to approximate $f$ accurately everywhere,  only where $f$ is small.

\begin{algoM} \label{AlgoM}
	For some finite interval $[a,b]$, integer $n_{\ninit} \ge 5$, and constant $\fC_0 \ge 1$, let
	$\fh$ and $\fC(h)$ be defined as in~\eqref{hCdef}.  Let $f:[a,b] \to \reals$ and
	$\abstol >0$ be user inputs. Let
	$n=n_{\ninit}$, and define the initial partition of equally spaced points, $x_{0:n}$, and
	certain index sets of subintervals:
\[
x_i=a+ i\frac{b-a}{n}, \ i \in \zton, \qquad \ci_{+} =  2\!:\!(n-1), \quad \ci_{-} =  1\!:\!(n-2).
\]
	Compute $\hM= \minfi$.   For $ s \in \{+,-\}$ do the following.
	
	\begin{enumerate}[\em Step 1.]
		
		\item \label{stagemin1} \emph{Check for convergence.}
		Compute $\oerr_i $ for all $i \in \ci_{\pm}$ according to \eqref{oerrdef}. Let
		$\widetilde{\mathcal{I}}_{s} = \left\{i \in \mathcal{I}_{s}: \oerr_i
		> \abstol \right\}$.  Next compute
		\begin{gather*}
		\herr_{i,s} := \oerr_i + \hM - \min\bigl(f(x_{i - s2}),f(x_{i-s1})\bigr) \quad
		\forall i \in \widetilde{\mathcal{I}}_{s}, 	\\
		\widehat{\ci}_s = \Bigl\{i \in \widetilde{\ci}_s:  \herr_{i,s} > \abstol \text{ or }
		\bigl( i-s3 \in \widetilde{\ci}_{-s} \text{ \& } \herr_{i-s3,-s} > \abstol \bigr) \Bigr\}.	
		\end{gather*}
		If $\widehat{\mathcal{I}}_+ \cup \widehat{\ci}_- =
		\emptyset$, return $M(f,\abstol) = \widehat{M}$ and terminate the algorithm.
		Otherwise, continue to the next step.
		
		\item \label{stagemin2} \emph{Split the subintervals as needed.}
		Update the present partition, $x_{0:n}$, to include the subinterval midpoints
		\begin{equation*}
		\frac{x_{i-s2} +x_{i-s1}}{2}, \ \frac{x_{i-s1} +x_{i}}{2} \quad \forall i \in \widehat{\ci}_s.
		\end{equation*}
    	(The point $(x_{i-2}+x_{i-1})/2$ is only included for $i \ge 2$, and the point
		$(x_{i+1} +x_{i+2})/2$ is only included for $i \le n-2$.)  Update the sets
		 $\ci_{\pm}$ to consist of the new indices corresponding to the old points
		\[
		x_{i-s1}, \ \frac{x_{i-s1} +x_{i}}{2} \text{ for } i \in \widehat{\ci}_s.
		\]
		(The point $x_{i-1}$ is only included for $i \ge 2$, and the point $x_{i+1}$ is only
		included for $i \le n-2$.) Return to Step~\ref{stagemin1}.
	\end{enumerate}
\end{algoM}

\begin{theorem} \label{thm:algMworks}
Algorithm $M$ defined above satisfies~\eqref{optprob} for functions in the
cone $\cc$ defined in~\eqref{conedef}.
\end{theorem}

\begin{proof}
	The proof of success of Algorithm $M$ is similar to that for Algorithm~$A$. Here
	we give the highlights.     We use the notation of $I_{x,l}$ introduced in \eqref{Ixldef}
	and analogously define $\tell(x)$ such that $I_{x,\tell(x)}$ is the final subinterval in
	Algorithm $M$ containing $x$ when the algorithm terminates.   For a fixed $x \in
	[a,b]$ we argue as in the proof of Theorem \ref{thm:algAworks} that there exist
	$l_{\pm} \le l_{*} \le \ell(x)$ such that $I_{x,l_*} =[x_{i_{l_*}-1},x_{i_{l_*}}]$,
	$x_{i_{l_{\pm}-1}} \le x_{i_{l_*}-1} \le x_{i_{l_*}} \le x_{i_{l_{\pm}}}$, and
	\begin{gather*}
	\frac{1} 8 \fC(3h_{l_-}) h_{l_-}^2 \norm[-\infty,{[x_{i_{l_-}-3},x_{i_{l_-}-1}]}]{f} + \hM_{l_*}
	-  \min\bigl(f(x_{i_{l_*}-1}),f(x_{i_{l_*}})\bigr)   \le \abstol , \\
	\frac{1} 8 \fC(3h_{l_+}) h_{l_+}^2 \norm[-\infty,{[x_{i_{l_+}},x_{i_{l_+}+2}]}]{f} +  \hM_{l_*}
	-  \min\bigl(f(x_{i_{l_*}-1}),f(x_{i_{l_*}})\bigr)   \le \abstol,
	\end{gather*}
	where $ \hM_{l}$ denotes the value of $\hM$ at iteration $l \in \natzero$.  By the
	definition of $\cc$ in~\eqref{conedef}, this then implies that
	\begin{align}
	\nonumber
		\MoveEqLeft[5]{\hM_{l_*} - \min_{x_{i_{l_*}-1} \le x \le x_{i_{l_*}}} f(x) } \nonumber
	\\ \le \ & \hM_{l_*}  - \minfii + \frac{1}{8} h_{l_*}^2\norm[{[x_{i_{l_*}-1}, x_{i_{l_*}}]}]{f} \le
	\abstol. \label{herrtogether}
	\end{align}
Further iterations of the algorithm can only make $\hM_{l}$ possibly closer to the
solution, $\min_{a \le x \le b} f(x) $.
\end{proof}

Figure \ref{f3foolplot}(b) displays the same function $-f_1$ as in
Figure \ref{f3foolplot}(a), but this time with the sampling points used for
minimization. Here $M(-f_1,0.02)$ uses only~$43$ points, whereas
$A(-f_1,0.02)$ uses $65$ points. This is because $-f_1$ does not need to be
approximated accurately when its value is far from the minimum.

\subsection{The Computational Cost of $M$} \label{subsec:optcost}
The derivation of an upper bound on the cost of Algorithm $M$ proceeds in a
similar manner as that for Algorithm $A$. There are essentially two reasons that a
subinterval $[x_{i-1},x_i]$ need not be split further. The first reason is the same as that
for Algorithm $A$: the function being minimized is approximated on
$[x_{i-1},x_i]$ with an error no more than the tolerance $\abstol$. This is
reflected in the definition of $\widetilde{\ci}_{\pm}$ in Step \ref{stagemin1} of Algorithm
$M$.  The second
reason is that, although the spline approximation error on $[x_{i-1},x_i]$ is larger
than $\abstol$, the function values on that subinterval are significantly larger
than the minimum of the function over $[a,b]$.  This is reflected in the definition of
$\widehat{\ci}_{\pm}$ in Step \ref{stagemin1} of Algorithm $M$.

Our definition of $\tL(x)$ reflects these two reasons. Let $x_*$ be some place
where the minimum of $f$ is obtained, i.e.,
$f(x_*)  = \min_{a \le x \le b} f(x).$
Let
\begin{equation}\label{eqn:defofltx}
\tL(x) = \min\bigl(L(x),\chL(x)\bigr), \qquad x \in [a,b],
\end{equation}
where $L(x)$ is defined above in \eqref{eqn:defoflx},
\begin{multline}\label{eqn:defoflt1x}
\chL(x) = \min \Biggl\{ l \in \mathbb{N}_0 :  \biggl\{ \biggl[\frac 18 \fC\left(3h_l\right) + 2
\biggr]   \norm[\tilde{I}_{x,l}]{f''} + \frac18\norm[{I_{x_*,l}}]{f''} \biggr\} h_{l}^2 \\ \qquad
\qquad +  2 \abs{f'(x)} h_{l} + [f(x_*) - f(x)] \le 0 \Biggr\},
\end{multline}
and $\tilde{I}_{x,l}$ is similar to $\Ixl$, but with generally seven times the width:
\begin{equation*}
\tilde{I}_{x,l}=\left[a+\max(0,j-4)h_l, a+ \min(j+3,2^l n_{\ninit})h_l\right] \supset \Ixl,
\end{equation*}
with the same $j$ as in \eqref{Ixldef} above.

Note that $\chL(x)$ does not depend on $\abstol$, whereas $L(x)$ does. As is the
case with $L(x)$, both $\chL(x)$ and $\tL(x)$ depend on $f$, although this
dependence is suppressed in the notation.

\begin{theorem}\label{thm:Mcost}
	Denote by $\cost(M,f,\abstol)$ the number of functional evaluations required by
	$M(f,\abstol)$. This computational cost is bounded as follows:
	\begin{equation*}
	\cost(M,f,\abstol) \le \frac{1}{h_0}\int_a^b 2^{\tL(x)} \, \dif x +1, \\
	\end{equation*}
	where $\tL(x)$ is defined in~\eqref{eqn:defofltx}.
\end{theorem}

\begin{proof}

Using the same argument as for Theorem \ref{thm:cost}, we only need to show that
$\tell(x) \le \tL(x)$ for all $x \in [a,b]$.  At each iteration of Algorithm $M$, the index sets
$\ci_{\pm}$ are both subsets of $\ci$ for the corresponding iteration of Algorithm $A$.
Thus $\tell(x)
\le L(x)$ by the same argument as used to prove Theorem \ref{thm:cost}.  We
only need to show that $\tell(x) \le \chL(x)$.
	
We will show that $\chL(x) < \tell(x) \le L(x)$ for any fixed $x$  leads to a contradiction.
If $\chL(x) < \tell(x)$, then at the $\chL(x)^{\text{th}}$  iteration,
$I_{x,\chL(x)} = [x_{i-1},x_i]$ for some $i$ must be split in Step~\ref{stagemin2} of $M$,
where $x_i-x_{i-1}= h_{\chL(x)} = h_0 2^{-\chL(x)}$.  This means that one or more of the
following must exceed $\abstol$:
\[
\herr_{i+2,+} , \ \herr_{i+1,+}, \ \herr_{i,+}, \  \herr_{i-1,-}, \ \herr_{i-2,-}, \ \herr_{i-3,-}.
\]
We prove that $\herr_{i+2,+} > \abstol$ is impossible.  The arguments for the other cases
are similar.

If $[x_{i-1},x_i]$ must be split because $\herr_{i+2,+} > \abstol$, then it is also the case
that $i-1 \in \widetilde{\ci}_-$, and so $\oerr_{i-1} > \abstol$.  In this case
\[
x_{j} - x_{j-1}
= h_{\chL(x)}   \mbox{ for } j=(i-1):(i+3).
\]
This means that $[x_{i-2},x_{i+3}] \in \tilde{I}_{x,l}$. By the same argument used in
\eqref{oerrbd} it can be shown that
\begin{equation}
\oerr_{i+2} \le \frac{1}{8} \fC(3 h_{\chL(x)}) h_{\chL(x)}^2 \norm[\tilde{I}_{x,\chL(x)}]{f''}.
\label{oerrupbd}
\end{equation}
The quantity $\oerr_{i+2}$ is the first term in the definition of $\herr_{i+2,+}$ in Step
\ref{stagemin1} of Algorithm $M$.

Next, we bound $\min\bigl(f(x_{i}),f(x_{i+1}) \bigr)$, which also appears in the
definition of $\herr_{i+2,+}$. As was argued earlier, $[x_{i-1},x_{i+1}] \in
\tilde{I}_{x,\chL(x)}$. Then a Taylor expansion about the arbitrary $x \in [x_{i-1},x_{i}] $
under consideration establishes that
\begin{equation} \label{minfjjm1bd}
\min\bigl(f(x_{i}),f(x_{i+1}) \bigr)
\ge f(x) - 2h_{\chL(x)}\abs{f'(x)} - h_{\chL(x)}^2 \norm[\tilde{I}_{x,\chL(x)}]{f''} .
\end{equation}
since $\abs{x_i - x} \le \abs{x_{i+1} - x} \le 2h_{\chL(x)}$.

Finally, we bound $ \hM_{\chL(x)}$. Let $x_*$ be a point where $f$ attains its minimum,
and
let $I_{x_*,l_*} = [x_{i_*-1}, x_{i_*}]$ be the subinterval in the present partition containing~$x_*$,
where $l_* \le \chL(x)$. By
\eqref{appxerrbda} it follows that
\begin{equation} \label{fxstarlowbd}
f(x_*) \ge \min(f(x_{i_*-1}),f(x_{i_*})) - \frac18 h_{l_*}^2\norm[I_{x_*,l_*}]{f''}.
\end{equation}
There are two possibilities regarding $l_*$. If $l_* < \chL(x)$, then by the argument in  in
\eqref{herrtogether} used to prove Theorem \ref{thm:algMworks},
\begin{align*}
\hM_{\chL(x)}   & \  \le \ \hM_{l_*}
\ \le \  \min(f(x_{i_*-1}),f(x_{i_*})) - \frac18 h_{l_*}^2 \norm[{[x_{i_*-1}, x_{i_*}]}]{f''}  + \abstol  \\
& \le  f(x_*) + \abstol \qquad \text{by}~\eqref{fxstarlowbd} .
\end{align*}
Otherwise, if $l_* = \chL(x)$, then
\begin{equation*}
\hM_{\chL(x)}   \le \min(f(x_{i_*-1}),f(x_{i_*}))  \le  f(x_*) + \frac18 h_{\chL(x)}^2
\norm[{I_{x_*,\chL(x)}}]{f''} \quad \text{by}~\eqref{fxstarlowbd} .
\end{equation*}
Thus, in either case we have
\begin{equation} \label{Mhatbd}
\hM_{\chL(x)} \le f(x_*) + \frac 18 h_{\chL(x)}^2  \norm[{I_{x_*,\chL(x)}}]{f''} + \abstol.
\end{equation}

Combining the three inequalities \eqref{oerrupbd}, \eqref{minfjjm1bd}, and
\eqref{Mhatbd} yields the inequality that allows us to contradict the assumption
that $\tell(x) > \chL(x)$:
\begin{align*}
\abstol & < \herr_{i+2,+} \qquad \text{by assumption} \\
& = \oerr_{i+2} + \hM_{\chL(x)} - \min\bigl(f(x_{i}),f(x_{i+1})\bigr) \qquad \text{by Step
\ref{stagemin1} of Algorithm $M$} \\
& = \frac{1}{8} \fC(3 h_{\chL(x)}) h_{\chL(x)}^2 \norm[\tilde{I}_{x,\chL(x)}]{f''} + f(x_*) +
\frac 18 h_{\chL(x)}^2  \norm[{I_{x_*,\chL(x)}}]{f''} + \abstol \\
& \quad  - f(x) + 2h_{\chL(x)}\abs{f'(x)} + 2h_{\chL(x)}^2
\norm[\tilde{I}_{x,\chL(x)}]{f''}  \qquad \text{by \eqref{oerrupbd}, \eqref{minfjjm1bd}, and
\eqref{Mhatbd}} \\
& \le \abstol + \biggl\{ \biggl[\frac 18 \fC\left(3 h_{\chL(x)} \right) + 2 \biggr]
\norm[\tilde{I}_{x,\chL(x)}]{f''} + \frac18\norm[{I_{x_*,\chL(x)}}]{f''} \biggr\} h_{\chL(x)}^2 \\
& \qquad \qquad +  2 \abs{f'(x)} h_{\chL(x)} + [f(x_*) - f(x)] \\
& \le \abstol \qquad \text{by \eqref{eqn:defoflt1x}}.
\end{align*}
This gives a contradiction and completes the proof.
\end{proof}

If $f(x)$ is close to the minimum function value, $f(x_*)$, for $x$ in much of
$[a,b]$, then $\chL(x)$ may be quite large, and $L(x)$  determines the
computational cost of Algorithm $M$. In this case, the computational cost for
minimization is similar to that for function approximation. However, if $f$ attains its
minimum at only a finite number of points, then for vanishing $\varepsilon$,
$\tL(x) = \widehat{L}(x)$ for nearly all $x$, and the computational cost for minimization
is significantly smaller than that for function approximation.

The minimization problem \eqref{optprob} for functions in the whole Sobolev
space $\cw^{2,\infty}$ has a similar lower complexity bound as \eqref{lowbdW}
for the function approximation problem by a similar proof. However, for
functions only in the cone $\cc$, we have not yet derived a lower bound on the
complexity of the minimization problem \eqref{optprob} for functions in $\cc$.

\section{Numerical Examples} \label{sec:examples}

Together with our collaborators, we have developed the Guaranteed Automatic
Integration Library (GAIL) \cite{ChoEtal15a}. This MATLAB software library
implements algorithms that provide answers to univariate and multivariate
integration problems, as well as \eqref{appxprob} and \eqref{optprob}, by
automatically determining the sampling needed to satisfy a user-provided error
tolerance. GAIL is under active development. It implements our best adaptive
algorithms and upholds the principles of reproducible and reliable computational
science as elucidated in Choi et al.~\cite{Cho13,Cho14a}. We have adopted practices including input parsing, extensive testing, code comments, a user
guide~\cite{Gail_ug}, and case studies. Algorithms $A$ and $M$ described here are implemented
as GAIL functions \funappxg{} and \funming, respectively in GAIL version 2.2. The following examples
showcase the merits and drawbacks of our algorithms. We compare them to the
performance of algorithms in MATLAB and the Chebfun toolbox.

Chebfun~\cite{TrefEtal16a} is a MATLAB toolbox that approximates functions in
terms of a Chebyshev polynomial basis, in principle to machine precision
($\approx 10^{-15}$) by default. In this example, we show that it fails to reach
its intended error tolerance for the function~$f_1$ defined in \eqref{f3def}
with $-c = 0.2 = \delta$.
Figure~\ref{f3chebfig}(a) shows the absolute errors of Chebfun's approximation to
$f_1$  with an input error tolerance $10^{-12}$, and the ``splitting'' option  turned on to allow
Chebfun to construct a piecewise polynomial interpolant if derivative
discontinuities are detected.  However, Chebfun produces some pointwise errors computed at a partition of $[-1,1]$ with even subinterval length $10^{-5}$ to be greater
than~$10^{-5}$.

In contrast, the pointwise errors of the piecewise linear interpolant produced
by \funappxg{} are uniformly below the error tolerance. Unfortunately, the time
taken by \funappxg{} is about $30$ times as long as the time
required by Chebfun.

\begin{figure}[tb]
\centering
\begin{tabular}{cc}
\includegraphics[width=5.7cm]{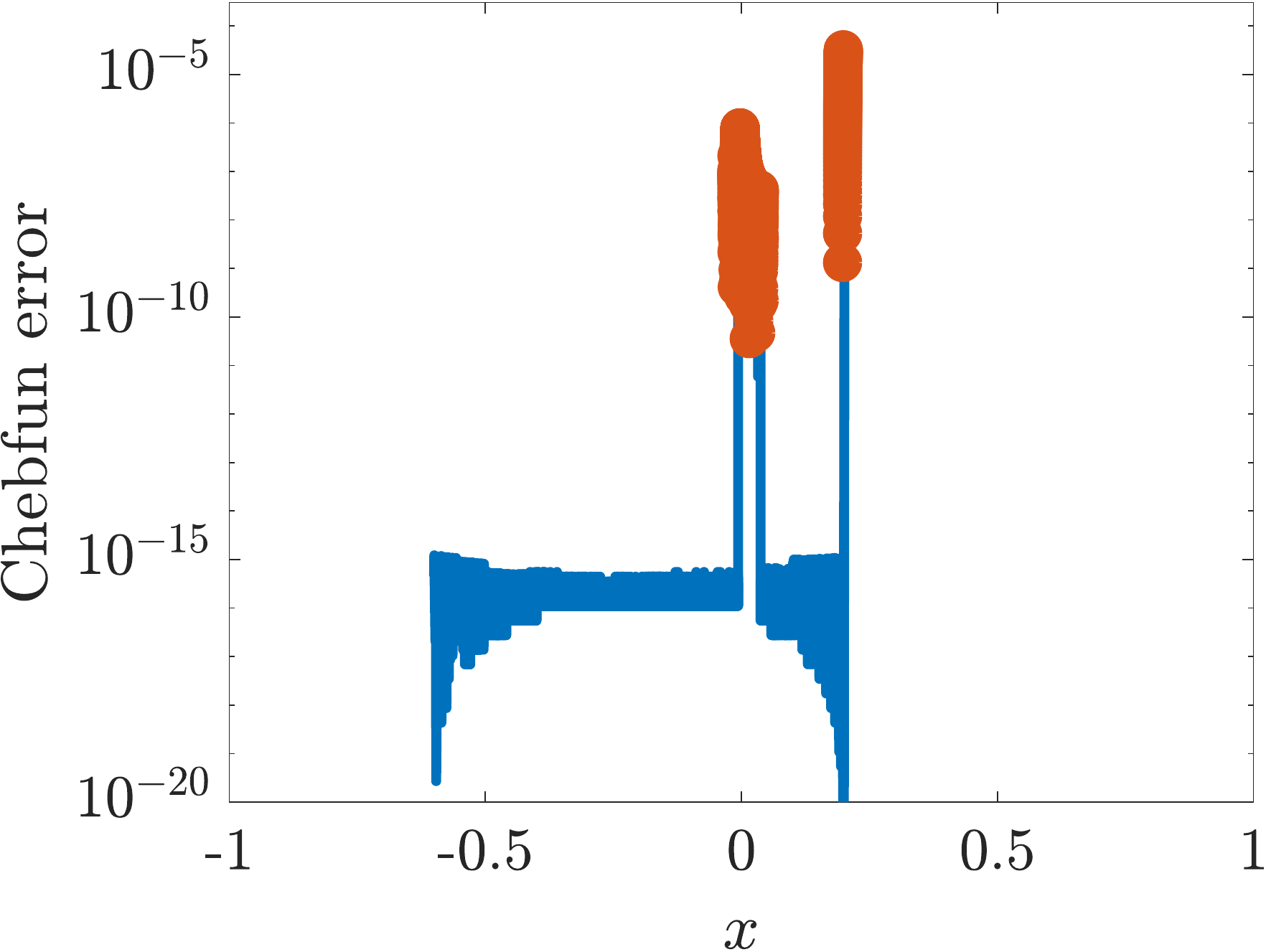} \hspace{-2.5ex} &
\includegraphics[width=5.7cm]{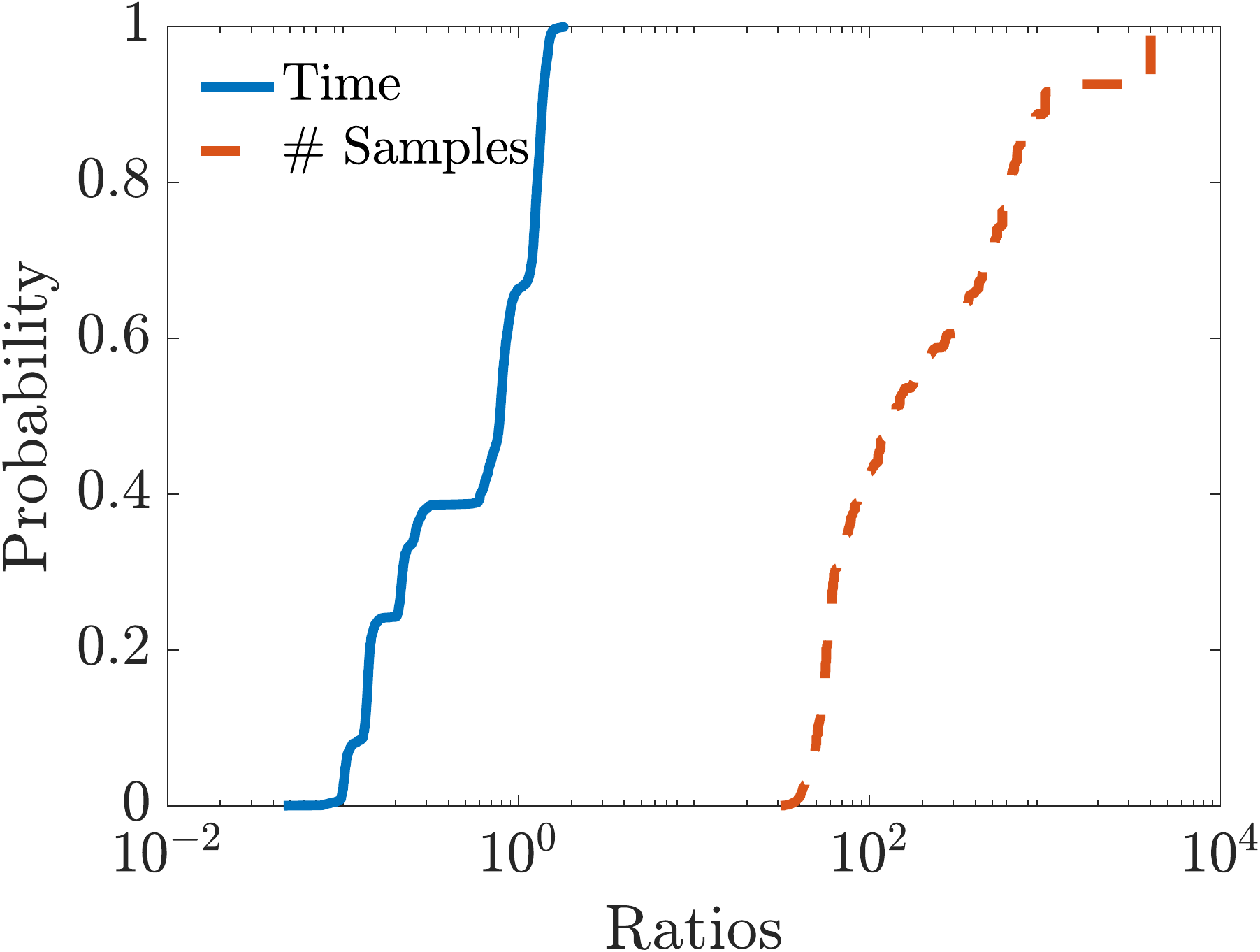}
\\ (a) & (b)
\end{tabular}
\caption{(a) The approximation errors for $f_1(x)$, $x \in [-1,1]$, with $-c = 0.2=\delta$
using  Chebfun with an error tolerance of $10^{-12}$.
(b) An empirical distribution function of performance ratios based on 1000
simulations for each test function in \eqref{eq:testfunctions}:
\funappxg{} time $/$ Chebfun time (solid), \funappxg{} \# of samples $/$
Chebfun \# of samples (dashed).
The data for this figure is conditionally reproducible
by
\texttt{funappx\_g\_test.m}  and  \texttt{LocallyAdaptivePaperFigs.m} in GAIL.
\label{f3chebfig}} 
\end{figure}



Next, we compare our adaptive algorithms with Chebfun  for random samples from the
following families of test functions defined on $ [-1, 1]$:
\begin{subequations} \label{eq:testfunctions}
\begin{align}
f_1(x) & \quad \text{defined in \eqref{f3def}}, \quad \delta = 0.2, \quad \ c \sim \cu[0,0.6], \\
f_2(x) &= x^4 \sin(d/x), \hspace{15.5ex} d \sim \cu[0,2], \\
f_3(x) &= 10  x^2 + f_2(x),
\end{align}
\end{subequations}
where  $\cu[a,b]$ represents a uniform distribution over $[a,b]$.
We set $n_{\ninit}= 250$, $\fC (h) =
10 \fh/(\fh-h)$, and $\abstol = 10^{-6}$. Our new algorithm \texttt{funappx\_g}
and Chebfun are
used to approximate $1000$ random test functions from each family. For Chebfun we
override the default tolerance to $10^{-6}$, and switch on the splitting
feature to allow piecewise Chebyshev polynomials for approximation.  Success  is
determined by whether a discrete approximation to the $L^{\infty}$ error is
no greater than the error tolerance.

\begin{table}[bt]
	\centering
	\caption{Comparison of number of sample points, computational time,  and success
		rates of \funappxg  
		and Chebfun in upper table;
		\funming, \fminbnd, and Chebfun's \texttt{min} in lower table.
		This table is conditionally reproducible by
		\texttt{funappx\_g\_test.m} and \texttt{funmin\_g\_test.m}  in
		GAIL.\label{tab:localVsGlobalVsChebfun}}
	\medskip
	{\footnotesize
		\setlength{\tabcolsep}{0.35em} 
		\begin{tabular}{crrrrrrrrrrrrrrrrrrrrrrrrrrrrrrrr}	
			\toprule	
			& \multicolumn{7}{c}{\bf Mean \# Samples}
			& \multicolumn{7}{c}{\bf Mean Time Used}
			& \multicolumn{7}{c}{\bf Success (\%)}
			\\ \midrule
			&
			\multicolumn{3}{c}{\funappxg}  & \multicolumn{3}{c}{Chebfun} &&
			\multicolumn{3}{c}{\funappxg}  &  \multicolumn{3}{c}{Chebfun} &&
			\multicolumn{3}{c}{\funappxg}  &   \multicolumn{3}{c}{Chebfun}
			\\ \toprule
			$f_1$  &&
			6557 &&& 116  &&&&	0.0029  &&& 0.0205  &&&& 100  &&& 0
			\\
			$f_2$  &&
			5017  &&&  43 &&&&  0.0031  &&& 0.0051 &&&& 100  &&& 3
			\\
			$f_3$  &&
			15698 &&& 22  &&&& 0.0049  &&& 0.0036  &&&& 100  &&&  3 	
			\\ \bottomrule
			 &
			\multicolumn{3}{c}{\funming} &  \multicolumn{3}{c}{\fminbnd}  &
			\multicolumn{1}{c}{\texttt{min}} &
			\multicolumn{3}{c}{\funming}  &  \multicolumn{3}{c}{\fminbnd }  &
			\multicolumn{1}{c}{\texttt{min} }  &
			\multicolumn{3}{c}{\funming} & \multicolumn{3}{c}{\fminbnd} &
			\multicolumn{1}{c}{\texttt{min}}
			\\ \toprule
			$-f_1$   &&
			111&&&  8 && 116 &&  0.0029 &&& 0.0006 && 0.0256 &&
			100 &&& 100 && 14
			\\ $\phantom{-}f_2$ &&
			48 &&& 22 && 43 && 0.0028 &&& 0.0007 && 0.0063 &&
			100 &&& 27 && 60
			\\ $\phantom{-}f_3$ &&
			108 &&& 9 && 22 && 0.0028 &&& 0.0007 && 0.0037 &&
			100 &&&  100 && 35
			\\ \bottomrule
		\end{tabular}
	}
\end{table}

We see in Table~\ref{tab:localVsGlobalVsChebfun} that \funappxg{} obtains
the correct answer in all cases,
even for $f_2$, which is outside the cone $\cc$. Since it is a higher order algorithm,
Chebfun generally uses substantially fewer samples than
\funappxg, but its run time is longer than \funappxg for a
significant proportion of the cases; see Figure~\ref{f3chebfig}(b). Moreover, Chebfun
rarely approximates the test functions satisfactorily.

Similar simulation tests have been run to compare our \funming, MATLAB's \fminbnd, and
Chebfun's \texttt{min}, but this time $n_{\ninit} = 20$ for \funming. The results are
summarized in the lower half of
Table~\ref{tab:localVsGlobalVsChebfun}. Our \funming{} achieves~100\%
success for all families of test functions with substantially fewer sampling
points and run time than \funappxg. This is because \funming does not
sample densely where the function is not close to its minimum value. Although MATLAB's
\fminbnd uses far fewer function values than \funming, it cannot locate the
global minimum (at the left boundary) for about 70\% of the $f_2$ test cases.
Chebfun's {\tt min} uses fewer points than \funming, but Chebfun is slower and
less accurate than \funming for these tests.


\section{Discussion}

Adaptive and automatic algorithms are popular because they require only a
(black-box) function and an error tolerance. Such algorithms exist in a
variety of software packages. We have highlighted those found in MATLAB and
Chebfun  because they are among the best. However, as we have shown
by numerical examples, these algorithms may fail. Moreover,
there is no theory to provide necessary conditions for failure, or equivalently,
sufficient conditions for success.

Our Algorithms $A$ (\funappxg) and $M$ (\funming) are locally adaptive and have
sufficient conditions for success. Although it may be difficult to
verify those conditions in practice, the theory behind these algorithms provide
several advantages:
\begin{itemize}
	
\item The cone, $\cc$, is intuitively explained as a set of functions whose  second
derivatives do not change drastically over a small interval.
This intuition can guide the user in setting the parameters defining~$\cc$, if
 desired.
	
\item The norms of $f$ and its derivatives appearing in the upper bounds of computational cost
in Theorems \ref{thm:cost} and \ref{thm:Mcost} may be unknown, but
these theorems explain how the norms influence the time required by
our algorithms.
	
\item Our Algorithm $A$ has been shown to be asymptotically optimal for the
complexity of the function approximation problem \eqref{appxprob}.
	
\end{itemize}

The minimum horizontal scale of functions in $\cc$ is roughly $1/n_{\ninit}$. The
computational cost of our algorithms is at least $n_{\ninit}$, but $n_{\ninit}$ is not a
multiplicative factor. Increasing $n_{\ninit}$ makes our new algorithms more robust, and it may
increase the minimum number of sample points and 
computational cost, if any, only mildly.

As mentioned in the introduction, there are general theorems providing sufficient
conditions under which adaption provides
no advantage. Our setting fails to satisfy those conditions because $\cc$ is not
convex. One may average two mildly spiky functions in $\cc$---whose spikes have
opposite signs and partially overlap---to obtain a very spiky function outside
$\cc$.

Nonadaptive algorithms are unable to solve \eqref{appxprob} or
\eqref{optprob} using a finite number of function values if the set of
interesting functions, $\cc$, is a cone, unless there exist nonadaptive
algorithms that solve these problems exactly. Suppose that some nonadaptive,
algorithm~$A$ satisfies \eqref{appxprob} for some cone~$\cc$, and that for an error
tolerance $\abstol$, this algorithm~$A$ requires $n$ function values. For any positive
$c$,
define $A^*(f,\abstol) = A(cf,\abstol)/c$ for all $f \in \cc$. Then $\norm{f -
A^*(f,\abstol)} = \norm{cf - A(cf,\abstol)}/c\le \abstol/c$ for all $f \in \cc$
since $cf$ is also in $\cc$. Thus, $A^*$ satisfies \eqref{appxprob} for error
tolerance $\abstol/c$, using the same number of function values as $A$.
Making $c$ arbitrarily large establishes the existence of a nonadaptive
algorithm that solves \eqref{appxprob} exactly.

Our algorithms do not take advantage of higher orders of
smoothness that the input function may have. We view the present work as a
stepping stone to developing higher order algorithms. Nonlinear splines
or higher degree polynomials, such as those used
in Chebfun, are potential candidates.

\section*{Acknowledgments}

We dedicate this article to the memory of our colleague Joseph F. Traub, who
passed away on August 24, 2015. He was a polymath and an influential figure in
computer science and applied mathematics. He was the founding Editor-in-Chief of
the Journal of Complexity and we are grateful to his tremendous impact and lifelong
service to our research community.

\begin{FJHchange}
	We thank Greg Fasshauer, Erich Novak, the GAIL team, and two anonymous referees
	for their valuable comments and suggestions. This research was supported in part by
	grants NSF-DMS-1522687 and NSF-DMS-0923111.
\end{FJHchange}


\bibliography{FJH23,FJHOwn23}

\end{document}